\documentclass[11pt]{amsart}
\usepackage{amsmath}
\usepackage{mathrsfs}
\usepackage{amsthm}
\usepackage{amssymb}%%%%%%%%%%%%%%%%%%%%%%%%%%%%%%%%%%%%%%%%%%%%%%%%%%%%%%%%%%%%%%%%%%%%%%%%%%%%%%%%%%%
\usepackage{amsfonts}

\newtheorem{theorem}{Theorem}

\newtheorem{lemma}[theorem]{Lemma}%%LemmaÓë Theorem ¹²Ïí±àºÅ

\theoremstyle{plain}

\newtheorem{proposition}[theorem]{Proposition}

\renewcommand\bigskip{\medskip}

\begin{document}
\title[]{ Generic Points of shift-Invariant Measures in the Countable Symbolic Space}
\author{Aihua Fan}
\address{Aihua Fan: LAMFA CNRS UMR 7352, Universit\'e  de Picardie Jules Verne, 33, Rue Saint Leu, 80039 Amiens Cedex 1, France}
\email{ai-hua.fan@u-picardie.fr}

\author{Mingtian Li }
\address{Mingtian Li: School of  Mathematics and Computer Science, Fujian Normal University, 350007 Fuzhou, People' Republic of China}
\email{limtwd@fjnu.edu.cn}

\author{Jihua Ma}
\address{Jihua Ma: Department of Mathematics Wuhan University, 430072 Wuhan, People' Republic of China}
\email{jhma@whu.edu.cn}

\maketitle
\begin{abstract}
We are concerned with sets of generic points for shift-invariant
measures in the countable symbolic space.  We measure the sizes of the sets by the Billingsley-Hausdorff dimensions defined by Gibbs measures. It is shown that the dimension of such a set is given by a variational
principle involving the convergence exponent of the Gibbs measure and the relative entropy dimension of the Gibbs measure with respect to the  invariant measure. This variational principle is different from that of the case of finite symbols, where the convergent exponent is zero and is not involved. An application is given to a class of   expanding interval dynamical systems.
\end{abstract}

 %% Ŀ¼
  %\tableofcontents

%%%%%%%%%%%%%%%%%%%%%%%%%%%%%%%
%%% ÕýÎIJ¿·Ö
%%%%%%%%%%%%%%%%%%%%%%%%%%%%%%%

%%ÕýÎÄ

\section{Introduction}
Consider the countable symbolic space $X=\mathbb{N}^\mathbb{N}$
endowed with the product topology and the shift mapping $T$ on $X$ defined by
 $$T(x_1x_2x_3\cdots)=(x_2x_3x_4\cdots).$$
 For any $T$-invariant Borel probability measure $\mu$ (we write $\mu\in\mathcal{M}(X,T)$), define the set of $\mu$-generic points by
 $$G_{\mu}:=\Big\{x\in X:\lim_{n\to\infty}\frac{1}{n}S_nf(x)=\int_{X} f\,\mathrm{d}\mu\quad\textrm{for all $f\in C_b(X)$}\Big\},$$
 where $S_nf(x):=\sum_{i=0}^{n-1}f(T^ix)$ is the $n$-th ergodic sum of $f$ and $C_b(X)$ denotes the space of all bounded real-valued continuous functions on $X$.

Our aim in this paper is to investigate the size of $G_\mu$ by
studying its Hausdorff dimension with respect to different metrics.
Let $\nu$ be another probability measure supported on the whole space $X$. It induces a metric $\rho_\nu$ on
$\mathbb{N}^\mathbb{N}$ as follows: if $x=y$, define
$\rho_\nu(x,y)=0$; otherwise
$$\rho_\nu(x,y)=\nu([x_1\cdots x_n]),$$ where $n=\inf\{k\ge 0:x_{k+1}\neq
y_{k+1}\}$ and $[x_1\cdots x_n]$ (called  cylinder) is the set of
all sequences having $x_1\cdots x_n$ as prefix. The Hausdorff
dimension of a subset of $X$ with respect to the metric $\rho_\nu$ is
the Billingsley dimension defined by $\nu$ (\cite{B}).

%The space $\mathbb{N}^\mathbb{N}$ is not compact. To some extent, a metric defined by a probability measure emphasizes bounded sequence %$x=(x_n)_{n\geqslant1}$. We could say that such a metric makes the sequence more or less ``compact".
In this paper, we only consider metrics defined by Gibbs measures
(see the definition of Gibbs measure in Section $2$. See also \cite{Sarig2}). Let
$\varphi:X\to\mathbb{R}$ be a function, called a potential. For any
$n\geqslant1$, we define its $n$-order variation by
 $${\rm{var}}_n\varphi:=  \sup\{|\varphi(x)-\varphi(y)|:x_i=y_i,~\textmd{for}~1\leqslant i\leqslant n \}.$$
 We say that $\varphi$  has summable variations if $$\sum_{n=2}^\infty {\rm {var}}_n\varphi<\infty.$$
The Gurevich pressure of a potential $\varphi$ with summable
variations is defined to be the limit
$$P_\varphi:=\lim_{n\to\infty}\frac{1}{n}\ln\sum_{T^nx=x}e^{S_n\varphi(x)}1_{[a]}(x),$$
where  $a\in \mathbb{N}$ (\cite{Sarig0}). It is shown that the limit exists and is
independent of $a$. It is known (\cite{Sarig1})
that a potential function $\varphi$ with summable variations  admits
a unique Gibbs measure $\nu$ iff ${\rm{var}}_1\varphi<\infty$ and
the Gurevich pressure $P_\varphi<\infty$.

  As we shall prove, the Billingsley dimension $ \dim_\nu G_ \mu$ is tightly related to the  convergence exponent of $\nu$, which is defined by
  $$\alpha_\nu:=\inf\Big\{t>0:\sum_{n=1}^{\infty}\nu([n])^t<+\infty\Big\}.$$ It is evident  that $\alpha_\nu\leqslant1$. We will prove that if the measure theoretic entropy $h_\nu$ is the infinity we have $\alpha_\nu=1$ (see Section 2.2).
  For $\mu\in\mathcal{M}(X,T)$, define  the (relative) entropy dimension of $\nu$ with respect to $\mu$ by
  $$\beta(\nu|\mu):=\limsup_{k\rightarrow\infty}\limsup_{N\rightarrow\infty}\frac{\sum_{\omega\in\Sigma^k_N}\mu([\omega])\ln\mu([\omega])}
  {\sum_{\omega\in\Sigma^k_N}\mu([\omega])\ln\nu([\omega])}$$
 where $\Sigma_N^k=\{1,\cdots,N\}^k$.
      Our main result is the following.

 \begin{theorem}\label{thm}
 Let $\mu\in \mathcal{M}(X,T)$ be an invariant Borel probability measure and  $\varphi$ be a potential function of summable  variations admitting a unique Gibbs measure $\nu$ with convergence  exponent $\alpha_\nu$. We have
 \begin{equation}\label{dimensionformula}
  \dim_\nu G_ \mu=\max\left\{\alpha_\nu,\beta(\nu|\mu)\right\}.
  \end{equation}
 \end{theorem}

Let us apply \eqref{dimensionformula} to two examples. First, when $\mu=\nu$, we have $\beta(\nu|\nu)=1$. Then $\dim_\nu G_\nu=1$. This can be obtained by the Birkhoff ergodic theorem, because $\nu$ is ergodic and of dimension $1$. Second, consider the invariant measure $\delta_x$ where $x=1^\infty$. It is easily seen that $\beta(\nu|\delta_x)=0$. Thus $\dim_\nu G_ {\delta_x}=\alpha_\nu$. This is not trivial.  It reflects the difference between the case of finite symbols and that of countable symbols.

A potential function
$\varphi$ admitting a Gibbs measure is upper bounded, so that the
integral $\int_X\varphi \, \mathrm{d}\mu$ is well defined as a number
in the interval $[-\infty,+\infty)$.
 For $\mu\in\mathcal{M}(X,T)$, define the relative entropy of $\nu$ with respect to $\mu$ by
$$h(\nu|\mu):=\limsup_{k\to\infty}-\frac{1}{k}\sum_{\omega\in\mathbb{N}^k}\mu([\omega])\ln\nu([\omega]).$$
As we shall see (see Proposition \ref{convergence}),
$$h(\nu|\mu)=P_\varphi-\int_X\varphi \, \mathrm{d}\mu.$$
 The dimension formula \eqref{dimensionformula} can be expressed by entropies. Suppose $\mu\neq\nu$. If $h(\nu|\mu)$ is finite, then  $\beta(\nu|\mu)=\frac{h_\mu}{h(\nu|\mu)}$
 (see Proposition \ref{alpha-beta}). Thus we have
$$\dim_\nu G_\mu=\max\left\{\alpha_\nu,\frac{h_\mu}{h(\nu|\mu)}\right\}.$$
If $h(\nu|\mu)$ is the infinity, we have $\alpha_\nu\geqslant \beta(\nu|\mu)$ (see Proposition \ref{alpha-beta}). It follows that
$$\dim_\nu G_\mu=\alpha_\nu.$$

 Let us present the idea of the proof. In the case of $\mu=\nu$, we have $\dim_\nu G_\nu=1$ because the Gibbs measure $\nu$ is ergodic and of dimension $1$. In the case of $\mu\neq\nu$, we first prove that  for any $\mu\in\mathcal{M}(X,T)$ there exists a sequence of ergodic Markov measures $\{\mu_j\}_{j\geqslant1}$ which converge in $w^*$-topology to $\mu$, and $h_{\mu_j}$ tends to $h_\mu$ whenever $h_\mu$ is finite.  Second, we show that $\alpha_\nu$ is a universal lower bound  by constructing a Cantor subset of $G_\mu$. For the other part of lower bound we distinguish two cases: in the case of $h(\nu|\mu)<+\infty$, we construct a subset of $G_\mu$ by using the sequence of ergodic Markov measures $\{\mu_j\}_{j\geqslant1}$ and show $\dim_\nu G_\mu\geqslant \frac{h_\mu}{h(\nu|\mu)} $; in the case of $h(\nu|\mu)=+\infty$, we
   show  $\alpha_\nu\geqslant\beta(\nu|\mu).$  For the upper bound, we adapt a standard argument
   by using an estimation  on the entropy of subword distribution which has combinatoric feather (see \cite{GW}).

In 1973, Bowen considered the set of generic points  $G_\mu$ in the setting
of topological dynamical system $T:X\to X$ over compact metric space $X$.
Bowen (\cite{bowen}) proved that for any $T$-invariant Borel
probability measure $\mu$ the topological  entropy of the set of generic
points  $G_\mu$ is bounded by the measure theoretic  entropy
$h_\mu$. Fan, Liao and Peyri\`ere (\cite{FLP}) showed that an
equality holds if $T$ satisfies the specification condition. In the
case of finite symbolic space, a study of the Billingsley dimension of  $G_\mu$ with respect to a shift-invariant Markov measure $\nu$ was performed by Cajar (\cite{cajar}). He proved that $\dim_\nu G_\mu$ is equal to the entropy dimension of $\nu$  with respect to $\mu$.  Olivier (\cite{Olivier}) extended this result to Billingsley dimension with respect to a shift-invariant $g$-measure. Furthermore, Ma and Wen (\cite{Mawen}) even showed that the Hausdorff and Packing measure of $G_\mu$ satisfy
a zero-infinity law. On the other hand, Gurevich and Tempelman (\cite{GT}) consider $G_\mu$ on high-dimensional finite symbolic systems. They evaluated the Hausdorff dimension of $G_\mu$ with respect to a wide class of metrics including Billlingsley metrics generated by Gibbs measures. Actually, there have been  many works done on
the  generic points set (\cite{cajar,FAN-FENG,GT,PS1,PS2}, see also  the
references therein).
 In the case of infinite symbolic space, the situation changes. Liao, Ma and Wang (\cite{LMW}) considered the set of continued fractions with maximal frequency oscillation. They proved that the set possessed  Hausdorff dimension $\frac{1}{2}$. This constant $\frac{1}{2}$ was first observed there. Fan, Liao and Ma (\cite{FLM}) considered sets of real numbers in $[0,1)$ with prescribed frequencies of partial quotients in their regular continued fraction expansions. They showed that $\frac{1}{2}$ is a universal lower bound of the Hausdorff dimensions of  these frequency sets. Furthermore,  Fan, Liao, Ma and Wang (\cite{FLMW}) considered the  Hausdorff dimension of Besicovitch-Eggleston subsets in countable symbolic space. They found that the dimensions  possess a universal lower bound depending only on the underlying metric.
  %Hence a natural guess is that, in the countable symbolic space, the Hausdorff dimension of generic points set $G_\mu$ for any %$\mu\in\mathcal{M}(X,T)$ shares the same property.

 Later, Fan, Jordan, Liao  and Rams (\cite{FJLR})  considered  expanding interval maps with infinitely many branches. They obtained  multifractal decompositions based on Birkhoff averages for a class of  continuous functions with respect to the Eulidean metric.
  %Note that for case of $s=1$ in Theorem \ref{cfs}, the Gibbs measure $\eta_1$ is just the Gauss measure which is boundedly equivalent to %Lebesgue measure on $[0,1)$. Hence, the case where $s=1$ in Theorem \ref{cfs} corresponds to Theorem $1.2$ in \cite{FJLR}.

Theorem \ref{thm} will be applied to study  the generic points of invariant measures in the Gauss dynamics which is related to the continued fractions.  Actually, our result can be applied to a class of expanding interval mapping system. Recall that the Gauss transformation $S:[0,1)\to[0,1)$ is defined by$$S(0):=0, ~ S(x):=\frac{1}{x}-\left\lfloor\frac{1}{x}\right\rfloor,~\forall x\in(0,1).$$
Let $\ell$ be an $S$-invariant Borel probability measure on $[0,1)$ and let $\mathcal{G}_\ell$ be the set of $\ell$-generic points. Consider the potential function $$\phi_s(x)=-s\ln|S'(x)|=2s\ln x$$ for $s>\frac{1}{2}$. The Gauss system is naturally coded by $\mathbb {N}^{\mathbb{N}}$.  It is  known that $\phi_s$ has summable variations and admits  a unique Gibbs measure $\eta_s$ whose convergence exponent is denoted by~ $\alpha_s$. By a standard technique of transferring dimension results from the symbolic space to the interval $[0,1)$, we  obtain the following result.
 \begin{theorem} \label{cfs}
                Let $\ell\in\mathcal{M}([0,1), S)$ be an $S$-invariant Borel probability measure and  $s>\frac{1}{2}$. If  $-\int_0^1\ln {x}\, \mathrm{d}\ell(x)<\infty$, then
                $$\dim_{\eta_s}\mathcal{G}_{\ell}=\max\left\{\alpha_s,
                \frac{h_{\ell}}{P_{\phi_s}-2s\int_0^1\ln {x}\, \mathrm{d}\ell(x)}\right\};$$
                otherwise, we have
                $$\dim_{\eta_s}\mathcal{G}_{\ell}=\alpha_s;$$

   \end{theorem}
Remark that the Gurevich pressure $P_{\phi_s}$ equals the infinity if $s\leqslant\frac{1}{2}$ and
  the case $s=1$ of Theorem \ref{cfs} corresponds to    Theorem $1.2$ in \cite{FJLR}.

The paper is organized as follows. In Section 2, we give some preliminaries.   Section 3 is devoted to the construction of a $\mu$-generic  point $x=(x_n)_{n\geqslant1}\in G_\mu$   satisfying $x_n\leqslant a_n$, where $\{a_n\}_{n\geqslant1}$ is a sequence of positive integers tending to the infinity. Using  this point as seed we construct a Cantor subset of $G_\mu$  and we obtain the lower bound for $\dim_\nu G_\mu$ in  Section 4. Section 5 is concerned with the upper bound for $\dim_\nu G_\mu$.
In Section 6, Theorem \ref{thm} is applied to   a class of expanding interval dynamics including the Gauss dynamics and Theorem 2 is proved there.

\section{Preliminaries }

In this section, we will make some preparations:  introducing  a metric to describe the $w^*$-convergence in $\mathcal{M}(X)$,  discussing Gibbs measures and defining the convergence exponent of a given measure,  approximating a $T$-invariant measure by ergodic Markov measures and  discussing the relative entropy of Gibbs measure with respect to a given $T$-invariant measure,  approximating the above mentioned Markov measure by orbit measures.

 First of all, let us begin with some notation.  We denote by $X$ the countable symbolic space $\mathbb{N}^\mathbb{N}$ endowed with the product topology and define the {\em shift  map} $T:X\to X$ by $$
 (Tx)_n=x_{n+1}.
   $$
   An element $(x_1\cdots x_n)\in\mathbb{N}^n$ is called an $n$-length {\em word}. Let $\mathcal{A}^*=\bigcup_{n=0}^\infty\mathbb{N}^n$ stand for the set of all finite  words, where $\mathbb{N}^0$ denotes the set of  empty word. Given $x=(x_1x_2\cdots)\in X$ and
 $m\geqslant n \geqslant1$,  $$x|_n^m=(x_n\cdots x_m)$$
  denotes a subword of $x$.   For  $\omega=(\omega_1\cdots \omega_n)\in\mathbb{N}^n$, the $n$-{\em cylinder} $[\omega]$ is defined by
  $$ [\omega]=
   \{x\in X:x|_1^n=\omega\}.
    $$
    We will denote by  $\mathcal{C}^n$  the set of all $n$-cylinders for $n\geqslant0$. There is a one-to-one correspondence between $\mathbb{N}^n$ and $\mathcal{C}^n$. Let $\mathcal{C}^*=\bigcup_{n=0}^\infty\mathcal{C}^n$ denote the set of all cylinders. For $j,N\geqslant1$ we will write  $$
    \Sigma_N^j=\{1,\cdots,N\}^j, \qquad \mathcal{C}_N^j=\{[\omega]:\omega\in\Sigma_N^j\}.
    $$
\subsection{Metrization of the $w^*$-topology }
Recall that $\mathcal{M}(X)$ denotes the set of Borel
probability measures on $X$. We endow $\mathcal{M}(X)$ with the $w^*$-topology induced by $C_b(X)$.
Let us introduce a metric to describe the $w^*$-topology of $\mathcal{M}(X)$.

For every cylinder $[\omega]\in\mathcal{C}^*$, we choose a positive number $a_{[\omega]}$ so that
$$\sum_{[\omega]\in\mathcal{C}^*} a_{[\omega]}=1,$$
where the sum is taken over all cylinders.
For $\mu,\nu\in \mathcal{M}(X),$ define
$$d^*(\mu,\nu)=\sum_{[\omega]\in\mathcal{C}^*}{a_{[\omega]}|\mu([\omega])-\nu([\omega])|}.$$
The following proposition shows that the metric $d^*$ is compatible with the $w^*$-topology of $\mathcal{M}(X)$.

\begin{proposition} \label{weak-d}
Let $\{\mu_n\}_{n\geqslant1}\subset \mathcal{M}(X)$ and $\mu\in \mathcal{M}(X)$.~ Then~$\mu_n$ converges in $w^*$-topology to $\mu$ if and only if  $\lim_{n\to\infty}d^*(\mu_n,\mu)=0.$
\end{proposition}

\begin{proof}
Suppose  $\mu_n$ converges in $w^*$-topology to $\mu$, so that $\lim_{n\to \infty}\mu_n([\omega])=\mu([\omega])$ for any $\omega \in \mathcal{A}^*$.
Let $\epsilon>0$. Since $(a_{[\omega]})_{[\omega]\in\mathcal{C}^*}$ is a probability on the set of all cylinders, there exists a large integer $K\geqslant1$ such that
 \begin{equation}\label{equa-fin}
 \sum_{[\omega]\notin\mathcal{D}_{K}}a_{[\omega]}\leqslant\frac{\epsilon}{4},
 \end{equation}
where $ \mathcal{D}_K=\{[\omega]\in\mathcal{C}_K^m:1\leqslant m\leqslant K\}.$ Since $\mathcal{D}_{K}$ is a finite set, we can find a positive integer $N\geqslant1$ such that for any $n\geqslant N$ we have
\begin{equation}\label{equa-fin2}
\sum_{[\omega]\in\mathcal{D}_{K}}|\mu_{n}([\omega])-\mu([\omega])|\leqslant\frac{\epsilon}{2}.
\end{equation}
Then, by (\ref{equa-fin}), (\ref{equa-fin2}) and the fact that $\mu_n([\omega])\leqslant1$, we have
\begin{eqnarray*}
d^*(\mu_{n},\mu)%&=&\sum{a_{[\omega]}|\mu_{n}([\omega])-\mu([\omega])|}\\
=(\sum_{[\omega]\notin\mathcal{D}_{K}}+
\sum_{[\omega]\in\mathcal{D}_{K}})\  {a_{[\omega]}|\mu_{n}([\omega])-\mu([\omega])|}\leqslant\epsilon
\end{eqnarray*}
for $n\geqslant N$.
Thus we have proved  $\lim_{n\to\infty }d^*(\mu_n,\mu)=0.$

Conversely, suppose $\lim_{n\to\infty }d^*(\mu_n,\mu)=0.$  This implies immediately that  for any $[\omega]\in\mathcal{C}^*$,
 $\mu_n([\omega])\longrightarrow \mu([\omega])$  as $n\to\infty$ . We finish the proof  by using the following lemma which can be found in (\cite{B},~p.17).

\begin{lemma} \label{cr}
Let $Y$ be a metric space and $\mathcal{A}$ be a family of subsets of $Y$.  Then $\mu_n$ converges in $w^*$-topology to $\mu$ if the following conditions are satisfied:
\begin{enumerate}
\item[(1)]~ $\mathcal{A}$ is closed under the  finite intersection,
 \item[(ii)] any open set $U$ can be written as~$U=\bigcup_{n=1}^{\infty}{A_n}$ with $A_n\in\mathcal{A}$,
 \item[(iii)]for any $A\in\mathcal{A}$,
$\lim_{n\to\infty}\mu_n(A)=\mu(A)$.

\end{enumerate}
\end{lemma}
The set of cylinders has the above properties of $\mathcal{A}$.   \end{proof}

 For any $x\in X$ and $n\geqslant1,$ define the orbit measure    $$\Delta_{x,n}:=\frac{1}{n}\sum_{i=0}^{n-1}\delta_{T^ix}.$$     By Proposition \ref{weak-d}, we can rewrite $G_\mu$ as
$$G_\mu=\{x\in X:\Delta_{x,n}\stackrel{d^*}{\longrightarrow}\mu~\textmd{as}~ n\to\infty \}.$$

The metric $d^*$ can be extended to finite symbolic measure space over $X$ and has the sub-linearity described in the following proposition which will be useful in the sequel.

\begin{proposition}   \label{aff}

For any  $\mu_1,\mu_2,\nu_1,\nu_2\in \mathcal{M}(X)$ and any $\alpha,\beta\in\mathbb{R}$, we have
$$d^*(\alpha\mu_1+\beta\mu_2,\alpha\nu_1+\beta\nu_2)\leqslant|\alpha| d^*(\mu_1,\nu_1)+|\beta| d^*(\mu_2,\nu_2).$$
\end{proposition}
\begin{proof}

It follows immediately from the definition of the metric $d^*$.
\end{proof}

The following proposition shows that two orbit measures approach each other, even uniformly, when
the two orbit approach each other (under the Bowen metric).

\begin{proposition} \label{measure uniform}
 The following equality  holds:
\begin{equation}
\lim_{n\to\infty}\sup_{x|^n_1=y|^n_1}d^*(\Delta_{x,n},\Delta_{y,n})=0.
\end{equation}

\end{proposition}
\begin{proof}First, observe that
$$\sum_{n=0}^\infty{A_n}=1,\quad
\mbox{\rm where} \ \  A_n:=\sum_{[\omega]\in\mathcal{C}^n} a_{[\omega]}.
$$
 By the sub-additivity of $d^*$ stated in Proposition \ref{aff}, we have
\begin{eqnarray*}
d^*(\Delta_{x,n},\Delta_{y,n})&\leqslant&\frac{1}{n}\sum_{i=0}^{n-1}{d^*(\delta_{T^ix},\delta_{T^iy})}.
\end{eqnarray*}
For any integer $N$ such that $n> N\geqslant1,$ we break the above sum into two parts: $\sum_{i=0}^{n-N-1}+\sum_{i=n-N}^{n-1}.$
 Assume that $x|_1^n=y|_1^n$. By the definition of the metric $d^*$, we have
 $$\sum_{i=0}^{n-N-1}d^*(\delta_{T^ix},\delta_{T^iy})\leqslant \sum_{i=0}^{n-N-1}\sum_{j=n-i+1}^\infty A_j\leqslant (n-N)\sum_{j=N+2}^\infty A_j.$$
Then by noting  the trivial  fact that $d^*(\delta_{T^ix},\delta_{T^iy})\leqslant1,$ we have
 \begin{eqnarray*}
  d^*(\Delta_{x,n},\Delta_{y,n})&\leqslant&\frac{n-N}{n}\sum_{j= {N+2}}^\infty{A_j}+\frac{N}{n}.
 \end{eqnarray*}
Letting $n$ then $N$ tend to the infinity, we finish the proof.
\end{proof}

\subsection{Gibbs measure}\label{gibbsmeasure} We use
 Gibbs measures  to induce  metrics on $X$. The following facts about Gibbs measures can be found in \cite{Sarig2}.

 Recall that for a function $\varphi:X\to\mathbb{R}$, called {\em potential function}, the
 {\em $n$-order variation} of $\varphi$ is defined by
 $${\rm{var}}_n\varphi:=\sup\{|\varphi(x)-\varphi(y)|:x,y\in X,~x|^n_1=y|^n_1\}.$$
 We say that  a potential $\varphi$ has {\em summable variations} if $$\sum_{n=2}^\infty {\rm{var}}_n\varphi<+\infty.$$ It is easy to see that a potential $\varphi$ with summable variations is uniformly continuous on $X$.
The {\em Gurevich pressure} of $\varphi$ with summable variations is defined to be the limit
$$P_\varphi:=\lim_{n\to\infty}\frac{1}{n}\ln\sum_{T^nx=x}e^{S_n\varphi(x)}1_{[a]}(x),$$
where $a\in \mathbb{N}$ and it can be shown that the limit exists and is independent of $a$ (see \cite{Sarig0}).

An invariant probability measure $\nu$ is called a  {\bf Gibbs measure} associated to  a potential function $\varphi$ if it satisfies the Gibbsian property: there exist
 constants $C>1$ and $P\in\mathbb{R}$ such that
\begin{equation}\label{gibbs1}
\frac{1}{C}\leqslant\frac{\nu([x_1x_2\cdots x_n])}{\exp({S_n{\varphi(x)}}-nP)}\leqslant C
\end{equation}holds for any $n\geqslant1$ and any $x\in X$.
It is known (\cite{Sarig1}) that a potential function $\varphi$ with summable variations  admits a unique Gibbs measure $\nu$ iff ${\rm{var}}_1\varphi<+\infty$ and the Gurevich pressure $P_\varphi<+\infty$.
Assume that $\varphi$ admits a unique Gibbs measure $\nu_\varphi$. Then the constant       $P$ in (\ref{gibbs1}) is  equal to the Gurevich pressure $P_\varphi$.
Let $\varphi^*=\varphi-P_\varphi,$ we have
$$P_{\varphi^*}=0~\textmd{and}~ \nu_{\varphi^*}=\nu_\varphi.$$
Hence, without loss of generality, we always suppose  $P_\varphi=0$ in the rest of this paper. A trivial fact is that   the Gibbsian property (\ref{gibbs1}) implies:
 \begin{equation}\label{p-bound}
 \forall x\in X,\quad\varphi(x)\leqslant\ln C.
 \end{equation}
It follows that the integral $\int_X{\varphi} \, \mathrm{d}\mu$ is defined as a number in $[-\infty,+\infty)$ for any probability measure $\mu$.
Also, the Gibbsian property implies the quasi Bernoulli property which will be exploited many times in the present paper.
  \begin{lemma} \label{joint}
Let $\nu$ be a Gibbs measure associated to potential $\varphi$. For any $k$ words $\omega_1,\cdots,\omega_k$, we have
$$C^{-(k+1)}\nu([\omega_1\cdots\omega_k])\leqslant \nu([\omega_1])\cdots\nu([\omega_k])\leqslant C^{k+1}\nu([\omega_1\cdots\omega_k]).$$
\end{lemma}

  For any $T$-invariant Borel probability measure $\mu$, define the {\em relative entropy} of $\nu$ with respect to $\mu$ by
$$h(\nu|\mu)=\limsup_{k\to\infty}-\frac{1}{k}\sum_{\omega\in\mathbb{N}^k}\mu([\omega])\ln\nu([\omega]).$$
  It is trivially true that $h(\mu|\mu)=h_\mu$.

  When $\nu$ is the Gibbs measure associated to $\varphi$,  the relative entropy $h(\nu|\mu)$ is   equal to the integral $-\int_X{\varphi} \, \mathrm{d}\mu$.

 \begin{proposition} \label{convergence}
  Assume that $\varphi$ has summable variations and admits a unique Gibbs measure $\nu$. Then for any
  invariant measure $\mu\in\mathcal{M}(X,T)$, we have
$$h(\nu|\mu)=\lim_{k\to\infty}-\frac{1}{k}\sum_{\omega\in\mathbb{N}^k}\mu([\omega])\ln\nu([\omega])
=-\int_X{\varphi} \, \mathrm{d}\mu.$$

  \end{proposition}
  \begin{proof} For each cylinder $[\omega]$, we arbitrarily choose a point
  $x'_\omega$  in $[\omega]$. Then
  for any $\lambda\in\mathcal{M}(X)$, let
  $$I_k(\lambda)=\sum_{\omega\in\mathbb{N}^k}\lambda([\omega])\varphi(x'_{\omega}).$$
In virtue of \eqref{p-bound}, the above infinite series is defined as a number in $[-\infty,+\infty)$. Furthermore, the convergence of the series implies the absolute convergence. Since $\varphi$ is uniformly continuous on $X$, for any $\lambda\in\mathcal{M}(X)$ we have
   \begin{equation}\label{Ceso}
   \int_X{\varphi} \, \mathrm{d}\lambda=\lim_{k\to\infty}I_k(\lambda)=\lim_{k\to\infty}\frac{1}{k}\sum_{i=1}^kI_i(\lambda).
   \end{equation}

  First, we assume that $\int_X{\varphi} \, \mathrm{d}\mu>-\infty.$
  For $k\geqslant1,$ by the Gibbsian property of $\nu$, we have
$$
\sum_{\omega\in\mathbb{N}^k}\mu([\omega])\ln\nu([\omega])\leqslant  \ln C+\sum_{\omega\in\mathbb{N}^k}\mu([\omega])S_k\varphi(x_\omega).
    $$
 Notice that
 $$
    S_k\varphi(x_\omega) \le [\varphi(x'_{\omega_1\cdots\omega_k})+{\rm var}_k\varphi]
            +\cdots + [\varphi(x'_{\omega_k})+ {\rm var}_{1}\varphi].
 $$
 It follows that
    \begin{eqnarray*}
    \sum_{\omega\in\mathbb{N}^k}\mu([\omega])S_k\varphi(x_\omega)
    &\le& \sum_{j=1}^{k}  {\rm var}_{j}\varphi + [\sum_{\omega_1\cdots\omega_k\in\mathbb{N}^k}\mu([\omega_1\cdots\omega_k])\varphi(x'_{\omega_1\cdots\omega_k})\\
    &+&\cdots+\sum_{\omega_k\in\mathbb{N}}\mu([\omega_k])\varphi(x'_{\omega_k})]\\
   &=& \sum_{j=1}^{k}  {\rm var}_{j}\varphi +\sum_{i=1}^kI_i(\mu).
 \end{eqnarray*}
  Finally we get
  $$\sum\limits_{\omega\in\mathbb{N}^k}\mu([\omega])\ln\nu([\omega])\le
   \ln C+  \sum_{j=1}^{\infty}  {\rm var}_{j}\varphi + \sum_{i=1}^kI_i(\mu).$$
   In the same way, we can also get the opposite inequality
  $$\sum\limits_{\omega\in\mathbb{N}^k}\mu([\omega])\ln\nu([\omega])\geqslant
  -\ln C- \sum_{j=1}^{\infty}  {\rm var}_{j}\varphi +\sum_{i=1}^kI_i(\mu).$$
  Thus
  $$\left|\frac{1}{k}\sum\limits_{\omega\in\mathbb{N}^k}\mu([\omega])\ln\nu([\omega])
  -\frac{1}{k}\sum_{i=1}^kI_i(\mu)\right|\leqslant\frac{\ln C + \sum_{j=1}^{\infty}  {\rm var}_{j}\varphi}{k}.$$
    By (\ref{Ceso}) we obtain
     $$h(\nu|\mu)=-\int_X{\varphi} \, \mathrm{d}\mu.$$

Now assume that $\int_X{\varphi} \, \mathrm{d}\mu=-\infty$. The similar argument works in combination with the following fact:  there exists a sequence of points
     $\{x'_{\omega_1\cdots\omega_k}\}$  such that
     $$\lim\limits_{k\rightarrow\infty}\sum\limits_{\omega\in\mathbb{N}^k}\mu([\omega])\varphi(x'_\omega)=-\infty.$$
   \end{proof}

By the concavity of the logarithm function $\ln$, it is easy to show
\begin{equation}\label{relativeentropy}
h(\nu|\mu)\geqslant h_\mu.
  \end{equation}
   By Proposition \ref{convergence}, we can rewrite  the  variational principle (\cite{Sarig2}, p. 86) in the following form

  \begin{equation}\label{variationalprinple}
  P_\varphi=\sup_{\mu\in\mathcal{M}(X,T)}\left\{h_\mu-h(\nu|\mu):h(\nu|\mu)<+\infty\right\}.
  \end{equation}
 %Recall that we assume that $P_\varphi=0$. It is known that when $h_\nu<\infty$ this implies $h_\mu<h(\nu|\mu)$ when $\mu\neq\nu$ (\cite{Sarig2}, p. 89).
Recall that we assume that $P_\varphi=0$. It is known that the supremum in the variational principle \eqref{variationalprinple} is  attained only by a Gibbs measure $\nu$ with $h_\nu<\infty$ if such Gibbs measure exits (\cite{Sarig2}, p. 89). It follows that when $\nu\neq\mu$, we have $h(\nu|\mu)>h_\mu$, which implies $h(\nu|\mu)>0$.

 %% It is known that the above supremum can be obtained only by the Gibbs measure $\nu$ with $h_\nu<+\infty$ if such Gibbs measure exists (\cite{Sarig2}, p. 89).

   Recall that a Gibbs measure $\nu$  induces a metric $\rho_\nu$ on $X$:  for any $x,y\in X$, if $x=y$,
we define $\rho_\nu(x,y)=0$; otherwise $$\rho_\nu(x,y)=\nu([x|^n_1]),$$ where $n=\min\{k\geqslant0:x_{k+1}\neq  y_{k+1}\}$.    One can show   that $\rho_\nu$ is a   ultrametric  and induces the product topology on $X$ since $\nu$ is non-atomic and has $X$ as its support.
Let $n\geqslant1$ be an integer. Define
$$\delta_n=\sup\{\nu([u]):u\in\mathbb{N}^n\}.$$
The following proposition means that the $\rho_\nu$-distance of two points uniformly tends to zero
when they approach each other in the sense of Bowen.  This property  will be used in the proof for the upper bound of $\dim_\nu G_\mu.$
\begin{proposition}\label{metricuniform}
For $\{\delta_n\}_{n\geqslant1}$ defined on the above, one has
$$\lim_{n\to\infty}\delta_n=0.$$
\end{proposition}
\begin{proof}
Suppose that $\lim_{n\to\infty}\delta_n=a>0.$ Note that $\delta_n$ is non-increasing. Then there exists a sequence of cylinders $[u_n]\in \mathcal{C}^n$ so that $$\nu([u_n])>a/2.$$
Observe that two cylinders either are disjoint or one is contained in the other.
Since $\nu$ is a probability measure, there exists a cylinder $[u_{n_1}]$ which intersects infinitely many cylinders $\{[u_{n_k}]:k\geqslant1\}\subset\{[u_n]:n\geqslant1\}.$ Actually, the cylinder $[u_{n_1}]$ contains every element of $\{[u_{n_k}]:k\geqslant1\}$. By the same argument one can choose a cylinder $[u_{n_2}]$ with $n_2>n_1$ which contains infinite elements of $\{[u_{n_k}]:k\geqslant1\}$. In this way, one  choose a sequence of decreasing cylinders $\{[u_{n_j}]\}_{j\geqslant1}$ so that
$$\forall j\geqslant1,~\nu([u_{n_j}])>a/2.$$
 This contradicts the fact that the Gibbs measure $\nu$ has no atom.
\end{proof}

Remark that   Proposition \ref{metricuniform} also holds for any non-atom finite  measure $\eta$   supported on $X$.

Now we introduce several exponents which will be related to the Hausdorff dimension of $G_\mu$. Define  the  {\em convergence exponent} of $\nu$ by
   $$\alpha_\nu:=\inf\Big\{t>0:\sum_{n=1}^\infty{\nu([n])^t<+\infty}\Big\}.$$ For any $\mu\in\mathcal{M}(X,T)$, define the {\em entropy dimension} of $\nu$  with respect to $\mu$ by
     $$\beta(\nu|\mu):=\limsup_{k\rightarrow\infty}\limsup_{N\rightarrow\infty}\frac{H_{k,N}(\mu,\mu)} {H_{k,N}(\nu,\mu)},$$
     where $$
     H_{k,N}(\nu,\mu):=-\sum_{\omega\in\Sigma^k_N}\mu([\omega])\ln\nu([\omega]).
     $$
We are going to show  that if $h(\nu|\mu)<+\infty$, then $\beta(\nu|\mu)=\frac{h_\mu}{h(\nu|\mu)}$;
if $h(\nu|\mu)=+\infty$, then $\beta(\nu|\mu)\leqslant \alpha_\nu$. But first, we remark that
the convergence exponent  $\alpha_\nu$ has the following property.
\begin{lemma} \label{exponent-bound}
 Let $\alpha_\nu$ be the convergence exponent of  Gibbs measure $\nu$ associated to a potential function $\varphi$.
Then for any $\epsilon>0$ there exist  constants $C_0$ and $M$ such that
 \begin{equation}\label{uniform-bound-ineq}\sum_{\omega\in\mathbb{N}^k}\nu([\omega])^{\alpha_\nu+\epsilon}\leqslant C_0M^k,~(\forall k\geqslant1).
 \end{equation}
 \end{lemma}
 \begin{proof}   Let  $\gamma={\alpha_\nu+\epsilon}.$  By the definition of convergence exponent of $\nu$, we have $$M_0:=\sum_{\omega\in\mathbb{N}}\nu([\omega])^{\gamma}<+\infty.$$
     By  the quasi Bernoulli property (Lemma \ref{joint}), one gets
       \begin{eqnarray*}
&&\sum_{\omega\in\mathbb{N}^k}\nu([\omega])^\gamma\leqslant C^{\gamma(k+1)}\sum_{\omega\in\mathbb{N}^k}\nu([\omega_1])^\gamma\cdots\nu([\omega_k])^\gamma=C^\gamma (C^\gamma M_0)^k.
 \end{eqnarray*}
  Taking $C_0=C^\gamma$ and $M=C^\gamma M_0$ completes the proof.
 \end{proof}

%The exponents $\nu_\alpha$ and $\beta(\nu|\mu)$ have the following relationship.

If $\mu=\nu$, it is clear that  $\beta(\nu|\mu)=1$. However, we have the following claim.
\begin{proposition} \label{alpha-beta}
Let $\mu\in\mathcal{M}(X,T)$ and  $\varphi$ be a potential function of summable variations. Assume that $\varphi$ admits a unique Gibbs measure $\nu$ with convergence exponent $\alpha_\nu$. If $\nu\neq\mu$ and $h(\nu|\mu)<+\infty$,  then
\begin{equation}\label{entropy dimension-inequality}\beta(\nu|\mu)=\frac{h_\mu}{h(\nu|\mu)};\end{equation}
if $h(\nu|\mu)=+\infty$, we have
\begin{equation}\label{alpha betainequality}
\beta(\nu|\mu)\leqslant \alpha_\nu.
\end{equation}

\end{proposition}

\begin{proof}
%Recall that we assume that $P_\varphi=0$. It is known that the supremum in the variational principle \eqref{variationalprinple} is  attained only by a Gibbs measure $\nu$ with $h_\nu<\infty$ if such Gibbs measure exits (\cite{Sarig2}, p. 89). It follows that when $\nu\neq\mu$, we have $h(\nu|\mu)>h_\mu$, which implies $h(\nu|\mu)>0$.%That means, for any $\mu\in\mathcal{M}(X,T)$ which is different with $\nu$ we have $h_\mu-h(\nu|\mu)<0$
%(we have assumed that $P_\varphi=0$). It follows that $0<h(\nu|\mu)\leqslant+\infty$.
Assume that $\nu\neq\mu$.  According to the analysis following \eqref{variationalprinple}, we have  $h(\nu|\mu)>0$.

  If $0<h(\nu|\mu)<+\infty$, it follows that  $h_\mu<+\infty$. By Proposition \ref{convergence},  we have $\beta(\nu|\mu)=\frac{h_\mu}{h(\nu|\mu)}$.

If $h(\nu|\mu)=+\infty$, then
\begin{equation}\label{infty}
\limsup_{k\to\infty}\limsup_{N\to\infty}\frac{1}{k}H_{k,N}(\nu,\mu)=+\infty.
\end{equation}
Assume that $\gamma>\alpha_\nu$. To prove the inequality \eqref{alpha betainequality} we shall use the following result (see \cite{walters}, p.\,217): let $t_1,\cdots, t_m$ be given real numbers. If $s_j>0$
and $\sum_{j=1}^ms_j=1$ then
\begin{equation}\label{convex-inequality}
\sum_{j=1}^ms_j(t_j-\ln s_j)\leqslant\ln (\sum_{j=1}^me^{t_j}).
\end{equation}
         For any fixed integers $N$ and $k$ which are large enough, there is a bijection $\pi:\{1,\cdots,N^k\}\to\mathcal{C}_N^k$. Applying (\ref{convex-inequality}) to
 $m=N^k+1,s_j=\mu(\pi(j)),$ for$~1\leqslant j\leqslant N^k$ and $s_m=\sum_{[\omega]\notin\mathcal{C}_N^k}\mu([\omega])$ and $t_j=\gamma\ln\nu(\pi(j)),$
 for $~1\leqslant j\leqslant N^k$ and $t_m=0$, we obtain

 \begin{eqnarray*}
 &&\gamma\sum_{j=1}^{N^k}\mu(\pi(j))\ln\nu(\pi(j))-\sum_{j=1}^{N^k}\mu(\pi(j))\ln\mu(\pi(j))\\
 &\leqslant&s_m\ln s_m+\ln(1+\sum_{\omega\in\Sigma_N^k}\nu([\omega])^\gamma).
   \end{eqnarray*}
   Therefore,
  $$\frac{H_{k,N}(\mu,\mu)}
  {H_{k,N}(\nu,\mu)}\leqslant\gamma+
  \frac{s_m\ln s_m}{H_{k,N}(\nu,\mu)}
  +\frac{\ln(1+\sum_{\omega\in\Sigma_N^k}\nu([\omega])^\gamma)}
  {H_{k,N}(\nu,\mu)}.$$
  By \eqref{uniform-bound-ineq} and (\ref{infty}), we finish the proof by letting $N \to\infty$
  and then $k \to\infty$.
 %$$\dim_HG_\mu\geqslant\limsup_{k\rightarrow\infty}\limsup_{N\rightarrow\infty}
\end{proof}

 As a direct corollary, it follows  that a Gibbs measure $\nu$ with $h_\nu=+\infty$ has convergence exponent  $\alpha_\nu=1$ because $\beta(\nu|\nu)=1$.

 \subsection{Approximation of invariant measure by  Markov measures}
  For any $\mu\in\mathcal{M}(X,T)$, we are going to construct a sequence of ergodic Markov measures $\{\mu_j\}_{j\geqslant1}$ which approximate  $\mu$ in  $w^*$-topology. Actually, the entropy $h_{\mu_j}$
  of $\mu_j$ also approaches the entropy $h_\mu$ of $\mu$ wherever $h_\mu<+\infty$.

Let $l\geqslant1$ be an integer. An $l$-Markov measure with state $S=\mathbb{N}$
is a measure $\upsilon\in\mathcal{M}(X)$ having the Markov property:
%\begin{equation}\label{markovproperty}
$$\forall n>l~\mathrm{and}~\forall \omega_1\cdots \omega_n\in\mathbb{N}^n,\quad \frac{\upsilon([\omega_1\cdots \omega_n])}{\upsilon([\omega_1\cdots \omega_{n-1}])}=\frac{\upsilon([\omega_{n-l}\cdots\omega_n])}
{\upsilon([\omega_{n-l}\cdots\omega_{n-1}])}.$$
%\end{equation}
   Given $\upsilon\in\mathcal{M}(X)$ and $l\geqslant2$, by a standard construction one can obtain an $(l-1)$-Markov measure $\upsilon_l$
  which coincides
  with $\upsilon$ on all $l$-cylinders (see \cite{FFW}).

  In particular, a $1$-Markov measure (we call  Markov measure for simplicity) can be obtained
  by a stochastic matrix ${\bf{P}}=(p_{ij})_{S\times S}$ and a probability vector ${\bf{p}}=(p_i)_{i\in S}$.
    For any $x_1\cdots x_n\in S^n$
  $$\upsilon([x_1\cdots x_n]):=p_{x_1}p_{x_1x_2}\cdots p_{x_{n-1}x_n}.$$
 The measure $\upsilon$ is $T$-invariant iff ${\bf{p}}$ is invariant with respect to $\bf{P}$ (i.e. ${\bf{p}}\bf{P}={\bf{p}}$). Assume that  $\bf{P}$ is a    primitive matrix. Then $\bf{P}$ is positive recurrent iff there is an invariant probability vector ${\bf{p}}$ on $S$ (see \cite{Sen}, p.177).

  Since the partition consisting of all $1$-cylinders is a generator, the entropy $h_\mu$ of any invariant measure
                      $\mu\in \mathcal{M}(X,T)$  can be expressed as (see \cite{walters})

  $$h_\mu=\lim_{n\to\infty}\frac{1}{j}\sum_{\omega\in \mathbb{N}^j}-\mu([\omega])\ln\mu([\omega]).$$
  Especially, if the entropy $h_\mu$ is finite we have (see \cite{FFW})
$$h_\mu=\lim_{j\to\infty}\sum_{\omega_1\cdots\omega_{j}\in\mathbb{N}^{j}}-\mu([\omega_1\cdots\omega_{j}])\ln\frac{\mu([\omega_1\cdots\omega_{j}])}
{\mu([\omega_1\cdots\omega_{j-1}])}.$$

        The following proposition states that there exists a sequence of ergodic Markov measures $\{\mu_j\}_{j\geqslant1}$ approximating $\mu$ in
       $w^*$-topology.

\begin{proposition} \label{markov-app}
For every $\mu\in\mathcal{M}(X,T),$ there exists a sequence of ergodic Markov measures $\{\mu_j\}_{j=1}^\infty$ such that
 $$w^*\textmd{-}\lim_{j\to\infty}\mu_j= \mu.$$ Furthermore, if $h_\mu<+\infty$, we have
$$\lim _{j\to\infty}h_{\mu_j}= h_\mu.$$
\end{proposition}
\begin{proof}
First, we assume that $\mu$ is supported on the whole space $X$, otherwise we can place $\mu$  by $\mu_\epsilon=(1-\epsilon)\mu+\epsilon\mu_0$, where $0<\epsilon<1$ and $\mu_0$ is a fixed $T$-invariant Borel probability measure supported on whole space  $X$.
Fix $j\geqslant3$ (the cases $j=1$ and $j=2$  will be treated separately),
 we consider the state space $S_j=\mathbb{N}^{j-1}$ and the probability vector ${\bf{p}}_j=(\mu([\omega]))_{\omega\in S_j}$.  Take a  stochastic matrix ${\bf{P}}_j=(p_{\omega_1\cdots\omega_{j-1},\theta_1\cdots\theta_{j-1}})_{S_j\times S_j}$, where
\begin{eqnarray*}
p_{\omega_1\cdots\omega_{j-1},\theta_1\cdots\theta_{j-1}}=
\left\{\begin{array}{lc}
\frac{\mu([\omega_1\cdots\omega_{j-1}\theta_{j-1}])}{\mu([\omega_1\cdots\omega_{j-1}])},\quad~\textmd{if}~\omega_2\cdots\omega_{j-1}
=\theta_1\cdots\theta_{j-2},\\
0, ~\quad\quad\quad\quad\quad\quad\textmd {otherwise}.
\end{array}\right.\\
\end{eqnarray*}
Then,  ${\bf{P}}_j$ and ${\bf{p}}_j$ determine a unique Markov measure $\upsilon_j$ by virtue of the Markov property.
It is easy to check that ${\bf{p}}_j{\bf{P}}_j={\bf{p}}_j$. This implies that  $\upsilon_j$  is $T_j$-invariant, where $T_j$ is the shift
map on symbolic space $S_j^\mathbb{N}$. Furthermore, one can show that the matrix ${\bf{P}}_j$ is primitive  because ${\bf{P}}_j^n>0$ for  any $n>j$, so ${\bf{P}}_j$ is
positive recurrent.  In fact, by a standard argument, we  can even show that $\upsilon_j$ is  strong-mixing with respect to $T_j$ (see \cite{walters}, p. 42).    Let $\mu_j$ be the $(j-1)$-Markov measure which coincides with $\mu$ on all $j$-cylinders by a standard construction.
 One can show that $(S_j^\mathbb{N},T_j,\upsilon_j)$ is isomorphic to
$(X,T,\mu_j)$. Hence $\mu_j$ is ergodic with respect to $T$.
For the cases of $j=1$ and $j=2$, we consider state spaces $S_1=S_2=\mathbb{N}$ and ${\bf{p}}_1={\bf{p}}_2=(\mu([i]))_{i\in \mathbb{N}}$. Take ${\bf{P}}_1=(p_{ij})_{\mathbb{N}\times \mathbb{N}}$, where
$$p_{ij}=\mu([j]),~\forall ~i,j\geqslant1,$$
and ${\bf{P}}_2=(p_{ij})_{\mathbb{N}\times \mathbb{N}}$, where      $$p_{ij}=\frac{\mu([ij])}{\mu([i])},~\forall ~i,j\geqslant1.$$
By the same argument as in the case  $j\geqslant3$, one  obtain two ergodic measures  $\mu_1$  (usually called Bernoulli measure)  and $\mu_2$ on $X$.

Now we  prove that $\mu_j$ converge  in $w^*$-topology to $\mu$. By Proposition \ref{weak-d} it is sufficient to show $d^*(\mu,\mu_j)\to0~\textmd{as}~j\to\infty$.
 In fact,
\begin{eqnarray*}
d^*(\mu,\mu_j)&=&\sum_{[\omega]\in\mathcal{C}^*}{a_{[\omega]}|\mu([\omega])-\mu_j([u])|}\leqslant\sum_{|\omega|\geqslant j+1}a_{[\omega]}\to0~\textmd{as}~ j\rightarrow\infty.
\end{eqnarray*}

At last, assume that $h_\mu<+\infty$. Recall that $\mu$ coincides $\mu_j$  with all $j$-cylinders. Then, we have
\begin{eqnarray*}
h_\mu&=&\lim_{j\to\infty}\sum_{\omega_1\cdots\omega_{j}\in\mathbb{N}^{j}}-\mu([\omega_1\cdots\omega_{j}])\ln\frac{\mu([\omega_1\cdots\omega_{j}])}
{\mu([\omega_1\cdots\omega_{j-1}])}\\
&=&\lim_{j\to\infty}\sum_{\omega_1\cdots\omega_{j}\in\mathbb{N}^{j}}-\mu_j([\omega_1\cdots\omega_{j}])\ln\frac{\mu_j([\omega_1\cdots\omega_{j}])}
{\mu_j([\omega_1\cdots\omega_{j-1}])}\\
&=&\lim_{j\to\infty}h_{\mu_j}.\\
\end{eqnarray*}

\end{proof}
In the sequel, the sequence $\{\mu_j\}_{j\geqslant1}$   constructed in the proof of Proposition \ref{markov-app} will be called sequence of Markov approximation of $\mu$. Now we  present the approximation property of the relative entropy $h(\nu|\mu)$ by $h(\nu|\mu_j)$.

\begin{proposition}\label{int-app}
 Let $\mu\in\mathcal{M}(X,T)$ and  $\{\mu_j\}_{j\geqslant1}$ be the sequence  Markov approximation of $\mu$.
 Assume that the potential $\varphi$ with summable variations admits a unique Gibbs measure $\nu$.
\begin{enumerate}
\item[(i)] If $h(\nu|\mu)<+\infty$, then $h(\nu|\mu_j)<+\infty$  for all $j\geqslant1$ and
$$\lim_{j\to\infty}h(\nu|\mu_j)= h(\nu|\mu).$$
\item[(ii)] If $h(\nu|\mu)=+\infty$, then$$h(\nu|\mu_j)=
+\infty,~\forall j\geqslant1.$$

\end{enumerate}
\end{proposition}

\begin{proof} Recall that for any $\lambda\in\mathcal{M}(X)$, let
  $$I_k(\lambda)= \sum_{\omega\in\mathbb{N}^k}\lambda([\omega])\varphi(x'_{\omega}),$$ where $x'_\omega$ is an arbitrarily chosen point in $[\omega]$.
By Proposition \ref{convergence}, for any $\mu\in\mathcal{M}(X,T)$ we have
 $$h(\nu|\mu)=-\int_X{\varphi} \, \mathrm{d}\mu=-\lim_{k\to\infty}I_k(\mu).$$
%\begin{enumerate}
%\item[(i)]
Suppose $h(\nu|\mu)<+\infty$.  For any $\epsilon>0$ there exists an integer  $N_1\geqslant1$ such that                                                                                                                                                                                                          $$\forall k >N_1,~\left|h(\nu|\mu)+I_k(\mu)\right|\leqslant\frac{\epsilon}{2}.$$
On the other hand, by the definition of variation we have the following estimate
\begin{eqnarray*}
I_k(\mu_j)&=&\sum\limits_{[\omega]\in\mathcal{C}^{k-1}}\sum\limits_{\omega_{k}=1}^\infty\mu_j([\omega\omega_{k}])\varphi(x'_{\omega\omega_{k}})\\
&\leqslant&\sum\limits_{[\omega]\in\mathcal{C}^{k-1}}\sum\limits_{\omega_{k}=1}^\infty\mu_j([\omega\omega_{k}])(\varphi(x'_{\omega})+\rm{var}_{k-1}\varphi)\\
%&=&\sum\limits_{\omega\in\mathbb{N}^{k-1}}\mu_j[\omega]\varphi(x_{\omega})+{\rm{var}}_{k-1}\varphi \\
&=&I_{k-1}(\mu_j)+{\rm{var}}_{k-1}\varphi.
\end{eqnarray*}
A obvious induction on $k$ gives
\begin{equation}\label{left integral}
\forall k>j,~I_k(\mu_j)\leqslant I_j(\mu_j)+\sum_{i=j}^{k-1}{\rm{var}}_i\varphi.
\end{equation}
In the same way, we can also get the opposite inequality
%\begin{equation}\label{right integral}
$$\forall k>j,~I_k(\mu_j)\geqslant I_j(\mu_j)-\sum_{i=j}^{k-1}{\rm{var}}_i\varphi.$$
%\end{equation}
Thus, noting a fact that $I_j(\mu_j)=I_j(\mu)$, we have
 %\begin{equation} \label{leftmarkov}
            $$\forall k>j,~ |I_k(\mu_j)-I_j(\mu)|\leqslant\sum_{i=j}^{k-1} {\rm{var}}_i\varphi.$$
            % \end{equation}
Since $\varphi$ is of summable variations, there exists a positive integer   $N_2$ depending on $\epsilon$ such that
$$\sum\limits_{n>N_2} {\rm{var}}_n\varphi\leqslant\frac{\epsilon}{2}.$$
Let $N=\max\{N_1,N_2\}.$ For any $j\geqslant N$  and $k\geqslant j+1$, we have
\begin{eqnarray*}
\left|I_k(\mu_j)+h(\nu|\mu)\right|
&\leqslant&|I_k(\mu_j)-I_j(\mu)|+\left|I_j(\mu)+h(\nu|\mu)\right|\\
&\leqslant&\sum\limits_{i=j}^k {\rm{var}}_i\varphi+\frac{\epsilon}{2}\leqslant\epsilon.
\end{eqnarray*}
Letting $k\to+\infty$, we get
$$|h(\nu|\mu_j)-h(\nu|\mu)|\leqslant\epsilon.$$
Then we have proved
$\lim_{j\rightarrow\infty}h(\nu|\mu_j)=h(\nu|\mu)$.

Suppose $h(\nu|\mu)=+\infty$. Fix an integer $k$ and by a similar argument as above we have
 $$\forall n>k,~I_k(\mu)\leqslant I_n(\mu)+\sum_{i=k}^{n-1}{\rm{var}}_i\varphi.$$
 When $n$ tends to $\infty$,  $I_n(\mu)$ tends to $-h(\nu|\mu)$. Then  $I_k(\mu)=-\infty.$

 For any $j\geqslant1$,  by $\eqref{left integral}$ we have
\begin{eqnarray*}
h(\nu|\mu_j)
=\lim_{k\rightarrow\infty}-I_k(\mu_j)
\geqslant -I_j(\mu)-\sum_{i=j}^{+\infty} {\rm{var}}_i\varphi=+\infty.
\end{eqnarray*}
 Thus we complete the proof.
%\end{enumerate}
\end{proof}

 At the end of this section, we derive further useful properties of $\{\mu_j\}_{j\geqslant1}$ which will play an important role in the proof for lower bound of the Hausdorff dimension of $G_\mu$.
 One of these properties is that most of orbit measures approach $\mu_j$, as consequence of the ergodicity of
 $\mu_j$.

\begin{proposition} \label{app}
 Let $\mu\in\mathcal{M}(X,T)$ and $\{\mu_j\}_{j\geqslant1}$ be the sequence of  Markov approximation of $\mu$. Then  there exist a sequence of Borel subsets
 $\{M_j\}_{j\geqslant1}$ and a sequence of increasing integers $\{m_j\}_{j\geqslant1}$  such that
$$\forall j\geqslant 1,~\mu_j(M_j)\geqslant1-\frac{1}{4^j},$$
 $$\forall x\in M_j,n\geqslant m_j,~ d^*(\mu_j,\Delta_{x,n})\leqslant\frac{1}{4^j},$$

$$\forall n\geqslant m_j,~ \sup_{x|^n_1=y|^n_1}d^*(\Delta_{x,n},\Delta_{y,n})\leqslant\frac{1}{j}.$$
Furthermore, if $\varphi$ has summable variations and admits a unique Gibbs measure satisfying  $h(\nu|\mu)<+\infty$, then we can even find a sequence of increasing integers
$\{b_j\}_{j\geqslant1}$    such that  for any $j\geqslant1$

  $$\forall x\in M_j,1\leqslant n\leqslant m_j,~\left|S_n\varphi(x)\right|\leqslant b_j, $$
$$\forall x\in M_j, n\geqslant m_j,~\left|\frac{1}{n}S_n\varphi(x)+h(\nu|\mu_j)\right|\leqslant\frac{1}{4^j}, $$
$$\forall x\in M_j,n\geqslant m_j,~\left|\frac{1}{n}\ln\mu_j([x|^n_1])+h_{\mu_j}\right|\leqslant\frac{1}{4^j}.$$

\end{proposition}
\begin{proof}

This is a rather direct consequence of the Egoroff theorem, the convergence results proved above, the Birkhoff ergodic theorem and Shannon-McMillan-Breiman theorem.

  First, by the ergodicity of $\mu_j$, for every cylinder $[\omega]$  we have
                $$\Delta_{n,x}([\omega])=\frac{1}{n}\sum_{i=0}^{n-1}\delta_{T^ix}([\omega])\to\mu_j([\omega]), ~\mu_j -a.e..$$
Hence by  Lemma \ref{cr}, it follows that
$$\lim_{n\rightarrow\infty}{d^*(\mu_j,\Delta_{x,n})}=0,~\mu_j- a.e..$$
By the Egoroff theorem, the above convergence is uniform on some Borel subset $M^{(1)}_j\subset X$ with $\mu_j(M^{(1)}_j)\geqslant1-\frac{1}{4^{j+1}}.$
 Therefore there exists  an increasing sequence of positive integers $\{m^{(1)}_j\}_{j\geqslant1}$ such that
$$d^*(\mu_j,\Delta_{x,n})\leqslant\frac{1}{4^{j}},\forall x\in M^{(1)}_j,n\geqslant m^{(1)}_j.$$

Second, since $h(\nu|\mu)<+\infty$, by  (\ref{relativeentropy}), Proposition \ref{markov-app} and \ref{int-app}
we have $$h_{\mu}<+\infty,~ h_{\mu_j}<+\infty~\textmd{and}~h(\nu|\mu_j)<+\infty,~\forall j\geqslant1.$$
 In this way, by  the Shannon-McMillan-Breiman theorem there exist a Borel subset $M^{(2)}_j\subset X$ and $m^{(2)}_j\in\mathbb{N}$ so that
$$\mu_j(M^{(2)}_j)\geqslant1-\frac{1}{4^{j+1}}$$
and $$\left|\frac{1}{n}\ln\mu_j([x|^n_1])+h_{\mu_j}\right|\leqslant\frac{1}{4^{j}},~\forall x\in M^{(2)}_j,n\geqslant m^{(2)}_j.$$

 Third, by the Birkhoff ergodic theorem there exist a Borel subset $M^{(3)}_j\subset X$ and $m^{(3)}_j\in\mathbb{N}$ such that
 $$\mu_j(M^{(3)}_j)\geqslant1-\frac{1}{4^{j+1}}$$
and $$\left|\frac{1}{n}S_n\varphi(x)+h(\nu|\mu_j)\right|\leqslant\frac{1}{4^j}, ~\forall x\in M^{(3)}_j, n\geqslant m^{(3)}_j.$$

By Proposition \ref{measure uniform}, there exists   positive integer $m_j^{(4)}\in\mathbb{N}$ such that for any $n\geqslant m_j^{(4)}$ and $x,y\in X$ satisfying $x|^n_1=y|^n_1$,
$$d^*(\Delta_{x,n},\Delta_{y,n})\leqslant\frac{1}{j}.$$
Take $m_j=\max_{1\leqslant i\leqslant4}\{m_j^{(i)}\}$. By the invariance of $\mu_j$, we have for any $n\geqslant1$, $$\left|\int_{X} S_n\varphi(x) \mathrm{d}\mu_j\right|<+\infty.$$
Thus there exist Borel
 subset $Q^{(n)}_j\subset X$ and  positive integer $b^{(n)}_j$ so that
 $$\mu_j(Q^{(n)}_j)\geqslant1-\frac{1}{2^n}\frac{1}{4^{j+1}},~n\geqslant1$$
  and$$\left|S_n\varphi(x)\right|\leqslant b^{(n)}_j,~\forall x\in Q^{(n)}_j,~n\geqslant1.$$
  Let $M^{(4)}_j:=\bigcap_{n=1}^\infty Q^{(n)}_j$ and $b_j:=\sum_{n=1}^{m_j}b^{(n)}_j.$ Thus we have

  $$\mu_j(M^{(4)}_j)\geqslant1-\frac{1}{4^{j+1}},$$and $$\left|S_n\varphi(x)\right|\leqslant b_j, ~\forall x\in M_j^{(4)},1\leqslant n\leqslant m_j.$$

        Finally, take $M_j:=\bigcap_{i=1}^4M^{(i)}_j$ and we complete the proof.
\end{proof}

%Let $N^*$ be the set of all finite word of $X$.

\section{The construction of a seed of $G_\mu$ }

 Let $\{a_n\}_{n\geqslant1}$ be  a sequence of positive integers tending to the infinity. In this section, we  construct a $\mu$-generic  point $z=(z_n)_{n\geqslant1}$ of $G_\mu$  such that $z_n\leqslant a_n$. By using such a point, called seed of $G_\mu$, we will construct a Cantor subset of $G_\mu$ in next section.  The techniques we use are inspired by \cite{Mawen}.

\begin{proposition}\label{seed}
For any  sequence of positive integers $\{a_n\}_{n\geqslant1}$ tending to the infinity, there exists a point $z=(z_n)_{n\geqslant1}\in G_\mu$ such that $z_n\leqslant a_n$ for all $n\geqslant 1$.

\end{proposition}
\begin{proof}
For simplicity, we first construct a $\mu$-generic point $z=(z_n)_{n\geqslant1}$ satisfying $z_n\leqslant n$ for $n\geqslant1.$

For $j\geqslant1$, let $\mu_j,m_j$ and $M_j$ be the same as in Proposition   \ref{app}.
 Given a sequence of finite words
$W_1,W_2,W_3,\cdots$,
juxtaposing the elements of the sequence, we get an infinite word$$z=W_1W_2W_3\cdots.$$
In what follows, we will define $\{W_j\}$ by induction on $j$.  By the way, a sequence of integers $\{n_j\}_{j\geqslant1}$ will also be defined by induction. When $n_j$ is defined, let
$$H_j=\{z_1\cdots z_{n_j}\in \mathbb{N}^{n_j}:z=(z_1\cdots z_{n_j}\cdots)\in M_j\}.$$

Take $n_1=m_2,~ n_2=m_3,~\widetilde{x}_1= 1^{n_1}$ and ~$\widetilde{x}_2\in H_2$. Denote the maximum of the digits in $\widetilde{x}_2$ by $s_2$. Then take a sufficiently large $t_1\in\mathbb{N}$ such that $N_1:=t_1 n_1\geqslant s_2$ and then define $$
W_1={\widetilde{x}_1}^{t_1}.
$$

Suppose $W_1,\cdots,W_j, N_1\cdots N_j$ are defined and $\widetilde{x}_1,\cdots,\widetilde{x}_{j+1},n_1\cdots,n_{j+1}$ are also defined. We  define $W_{j+1},N_{j+1}$ and $\widetilde{x}_{j+2},n_{j+2}$ in the following way: take $n_{j+2}\geqslant m_{j+3}^2$ and $\widetilde{x}_{j+2}\in H_{j+2}$. Denote the maximal number of the digits in $\widetilde{x}_{j+2}$ by $s_{j+2}$. Then take a $t_{j+1}\in\mathbb{N}$ such that $$N_{j+1}:=t_{j+1} n_{j+1}+N_j\geqslant\max\{ N^2_j,s_{j+2}\}.$$ Then define $$
W_{j+1}={\widetilde{x}_{j+1}^{t_{j+1}}}.$$

Let $z=(z_n)_{n\geqslant1}$ be the sequence $W_1W_2\cdots$. We show that $z_n\leqslant n$ for all $n\geqslant1$. In fact, for $N_j<n\leqslant N_{j+1}$, by the definition of $W_j$ we have
 $$z_n\leqslant s_{j+1}\leqslant N_j< n.$$

Now, we  show $z\in G_\mu$. Since the point $z$ is obtained by juxtaposing the prefixes of points with orbit measures approximating the sequence of  Markov approximation of $\mu$, it is natural to approximate the orbit measure $\Delta_{z,n}$ by this sequence of  Markov measures.   Suppose $n\in[N_j,N_{j+1})$ for some $j$. Recall that
$$N_j=t_1n_1+t_2n_2+\cdots+t_jn_j.$$ Then there exists a unique integer $t ~(0\leqslant t<t_{j+1})$ such that $$N_j+tn_{j+1}\leqslant n<N_j+(t+1)n_{j+1}.$$
Thus by the definition of $\Delta_{z,n}$ we divide it into four parts

\begin{eqnarray*}
\Delta_{z,n}&=&\frac{1}{n}(\sum_{i=0}^{N_{j-1}-1}+\sum_{i=N_{j-1}}^{N_{j}-1}+\sum_{i=N_{j}}^{N_j+tn_{j+1}-1}
+\sum_{i=N_j+tn_{j+1}}^{n-1})(\delta_{T^iz})\\
%&=&\frac{N_{j-2}}{n}\Delta_{z,N_{j-2}}+\frac{t_{j-1}n_{j-1}}{n}\Delta_{T^{N_{j-2}}z,{t_{j-1}n_{j-1}}}
%$+\frac{t_{j}n_{j}}{n}\Delta_{T^{N_{j-1}}z,{t_{j}n_{j}}}\\
%&+&\frac{tn_{j+1}}{n}\Delta_{T^{N_{j}}z,tn_{j+1}}+\frac{n-N_j-tn_{j+1}}{n}\Delta_{T^{N_{j}+tn_{j+1}}z,{n-N_{j}-tn_{j+1}}}\\
&=&\alpha_1\Delta_{z,N_{j-1}}+\alpha_2\Delta_{T^{N_{j-1}}z,{t_jn_j}}+\alpha_3\Delta_{T^{N_{j}}z,tn_{j+1}}\\
&+&\alpha_4\Delta_{T^{N_{j}+tn_{j+1}}z,{n-N_{j}-tn_{j+1}}},
\end{eqnarray*}
where
$$\alpha_1:=\frac{N_{j-1}}{n},~~\alpha_2:=\frac{t_{j}n_{j}}{n},\alpha_3:=\frac{tn_{j+1}}{n},~~\alpha_4:=\frac{n-N_j-tn_{j+1}}{n}.$$
 %\begin{eqnarray*}
%\alpha_1&=&\frac{N_{j-2}}{n}\\
%\end{eqnarray*}
It is obvious that $\sum_{i=1}^4{\alpha_i}=1.$
By Proposition \ref{aff}, it follows that
\begin{eqnarray*}
&&d^*\left(\alpha_1\mu_{j-1}+\alpha_2\mu_j+(\alpha_3+\alpha_4)\mu_{j+1},\Delta_{z,n}\right)\\
&\leqslant&\alpha_1+\alpha_2d^*(\mu_j,\Delta_{T^{N_{j-1}}z,{t_jn_j}})+\alpha_3d^*(\mu_{j+1},\Delta_{T^{N_{j}}z,tn_{j+1}})\\&+&\alpha_4d^*(\mu_{j+1},\Delta_{T^{N_{j}+tn_{j+1}}z,{n-N_{j}-tn_{j+1}}}).
%&\longrightarrow&0,j\longrightarrow\infty.
\end{eqnarray*}
We will  estimate the first three  terms  in the right-hand side of the above inequlaity in advance. By the definition of $N_j$, we have

$$\alpha_1=\frac{N_{j-1}}{n}\leqslant\frac{N_{j-1}}{N_j}\leqslant\frac{1}{N_{j-1}}.$$
Recalling the construction of $H_{j}$, there exists $y\in M_{j}$ such that $y|_1^{n_{j}}=\widetilde{x}_{j}$. Note that $T^{N_{j-1}}z\in M_j$. By Proposition \ref{app} we have
$$d^*(\mu_{j},\Delta_{y,n_{j}})\leqslant\frac{1}{4^{j}}~\textmd{and}~d^*(\Delta_{y,n_{j}},\Delta_{T^{N_{j-1}}z,n_{j}})\leqslant\frac{1}{j}.$$
By the sub-affinity of the metric $d^*$, it is easy to show $$d^*(\Delta_{y,n_{j}},\Delta_{T^{N_{j-1}}z,{t_{j}n_{j}}})\leqslant\frac{1}{j}.$$
Hence

\begin{eqnarray*}
d^*(\mu_{j},\Delta_{T^{N_{j-1}}z,{t_{j}n_{j}}})&\leqslant&d^*(\mu_{j},\Delta_{y,n_{j}})
+d^*(\Delta_{y,n_{j}},\Delta_{T^{N_{j-1}}z,{t_{j}n_{j}}})\\
&\leqslant&\frac{1}{4^{j}}+\frac{1}{j}.
\end{eqnarray*}
Similar arguments as above yield

%\begin{equation}\label{decreasing 2}
%$$d(\mu_j,\Delta_{T^{N_{j-1}}z,{t_jn_j}})\leqslant\frac{1}{j}+\frac{1}{4^j},$$
%\end{equation}

%\begin{equation}
$$d^*(\mu_{j+1},\Delta_{T^{N_{j}}z,tn_{j+1}})\leqslant\frac{1}{4^{j+1}}+\frac{1}{j+1}.$$
%\end{equation}
For the last term, we need to deal with two cases separately.

 Case 1: $N_j+tn_{j+1}\leqslant n<N_j+tn_{j+1}+m_{j+1}$.
In this case, we have
  $$\alpha_4=\frac{n-N_j-tn_{j+1}}{n}\leqslant\frac{m_{j+1}}{n}
\leqslant\frac{m_{j+1}}{n_j}\leqslant\frac{1}{m_{j+1}}.$$

Case 2: $N_j+tn_{j+1}+m_{j+1}\leqslant n<N_j+(t+1)n_{j+1}$.
In this case, similar arguments as on the above yield
$$d^*(\mu_{j+1},\Delta_{T^{N_{j}+tn_{j+1}}z,{n-N_{j}-tn_{j+1}}})\leqslant\frac{1}{4^{j+1}}+\frac{1}{j+1}.$$
Then we have
\begin{eqnarray*}
&&d^*\left(\alpha_1\mu_{j-1}+\alpha_2\mu_j+(\alpha_3+\alpha_4)\mu_{j+1},\Delta_{z,n}\right)\\
%&\leqslant&\alpha_1+\alpha_2 d(\mu_{j-1},\Delta_{T^{N_{j-2}}z,{t_{j-1}n_{j-1}}})+\alpha_3d(\mu_{j},\Delta_{T^{N_{j-2}}z,{t_{j-1}n_{j-1}}})\\
%&+&\alpha_4d(\mu_{j+1},\Delta_{T^{N_{j}}z,tn_{j+1}})+\alpha_5d(\mu_{j+1},\Delta_{T^{N_{j}+tn_{j+1}}z,{n-N_{j}-tn_{j+1}}})\\
&\leqslant&\frac{1}{N_{j-1}}+(\frac{1}{4^j}+\frac{1}{j})+(\frac{1}{4^{j+1}}+\frac{1}{j+1})+(\frac{1}{m_{j+1}}+\frac{1}{4^{j+1}}+\frac{1}{j+1})\\
&\to&0~\textmd{as} ~j\to\infty.
\end{eqnarray*}
This combined with $\lim_{j\rightarrow\infty}d^*(\mu,\alpha_1\mu_{j-1}+\alpha_2\mu_j+(\alpha_3+\alpha_4)\mu_{j+1})=0,$  yields
$$\lim_{n\rightarrow\infty}d^*(\mu,\Delta_{z,n})=0,$$
which implies $z\in G_{\mu}$.

Noting that the integer $t_j$ for defining  $N_j~(j\geqslant1)$  can be taken arbitrarily large, we can construct a $\mu$-generic point $z=(z_n)_{n\geqslant1}$ such that $z_n\leqslant a_n$ for all $n\geqslant1$.
\end{proof}

\section{The lower bound of the Hausdorff dimension of $G_\mu$}
 In this section, we will prove the lower bound for the Hausdorff dimension of $G_\mu$. It is well known that the Gibbs measure $\nu$ is ergodic (\cite{Sarig2}, p.99). If $\nu=\mu$, by Birkhoff Ergodic Theorem and Lemma \ref{cr} we have $\nu(G_\nu)=1.$ Then $\dim_\nu G_\nu=1.$ Since $\beta(\nu|\nu)=1,$ we have finished the proof of Theorem \ref{thm}  in the case of $\nu=\mu.$ In the sequel, we  consider the case where $\mu\neq\nu.$

First, we  prove that $\alpha_\nu$ is a  lower bound.
\begin{proposition} \label{alpha-upperbound}
Let $\mu\in\mathcal{M}(X,T)$ be an invariant Borel probability measure.
 Assume that $\varphi$ has summable variations and admits a unique Gibbs measure $\nu$ with convergence exponent $\alpha_\nu$. Then  we have

$$\dim_\nu G_\mu\geqslant\alpha_\nu.$$
\end{proposition}

To prove this lower bound, we will
construct a Cantor subset of $G_\mu$ and apply  Billingsley's theorem. Before so
doing, we  need the following lemma about the convergence exponent $\alpha_{\nu}$, which can be found in \cite{FLMW}. An ordering of cylinders is considered according to the
sizes of  $\nu$-measures of the $1$- cylinders.

\begin{lemma} \label{m-con-exp}
Let $\alpha_{\nu}$ be the  convergence exponent of $\nu$. Consider a bijection $\pi:\mathbb{N}\rightarrow \mathbb{N}$ such that $\nu([\pi(1)])\geqslant\nu([\pi(2)])\geqslant\cdots.$  Then there exists an increasing sequence $\{s_k\}_{k\geqslant1}\subset\mathbb{N}$ such that
\begin{equation}\label{eq31}
\sum_{k=1}^n\ln {s_k}\gg n^3,
\end{equation}
and for any $\varepsilon>0$ and any $0<\delta<1$, there exists an integer $N(\varepsilon,\delta)$ such that for $k\geqslant N$ we have
\begin{equation}\label{eq32}
s_k-s_k^\delta>s_k^{1-\varepsilon},
\end{equation}
and for $\pi^{-1}(\omega)\in(s_k-s_k^\delta,s_k]$, we have
\begin{equation}\label{eq33}
({1-\varepsilon})\ln\pi^{-1}(\omega)<{-\alpha_{\nu}}\ln\nu([\omega])<
\frac{1+\varepsilon}{1-\varepsilon}\ln\pi^{-1}(\omega).
\end{equation}
\end{lemma}

\begin{proof}
Since the  convergence exponent of $\nu$ is invariant under the bijection $\pi$, we have (see \cite{PS}, p. 26)
\begin{equation}\label{eq34}
\alpha_{\nu}=\limsup_{n\rightarrow\infty}\frac{\ln n}{-\ln \nu([\pi(n)])}.
\end{equation}
Therefore, one may choose  an increasing sequence $\{s_k\}_{k\geqslant1}$ satisfying $\eqref{eq31}$ such that the limit in \eqref{eq34} along with $s_k$ exists. So, for any
$0<\varepsilon,\delta<1,$ there exists an integer $N_1=N_1(\varepsilon,\delta)$ such that for all $k\geqslant N_1,$ $\eqref{eq32}$ is satisfied and
\begin{equation}\label{zz}
({1-\varepsilon})\ln s_k<-\alpha_{\nu}\ln\nu([\pi(s_k)])<({1+\varepsilon})\ln s_k.
\end{equation}
On the other hand, by $\eqref{eq34}$, there exists an integer $N_2$ such that for all $n\geqslant N_2$
\begin{equation}\label{eq35}
-\alpha_{\nu}\ln\nu([\pi(n)])>(1-\varepsilon)\ln n.
\end{equation}
Thus if we take $N=\max\{N_1,N_2\}$  then by$\eqref{eq32}$, $\eqref{zz}$ and $\eqref{eq35}$, for all $n\in(s_k-s_k^\delta,s_k](k\geqslant N)$, we have
 \begin{eqnarray*}
(1-\varepsilon)\ln n&<&-\alpha_{\nu}\ln\nu([\pi(n)])\\
&\leqslant&-\alpha_{\nu}\ln \nu([\pi(s_k)])\\
&<&(1+\varepsilon)\ln s_k<\frac{1+\varepsilon}{1-\varepsilon}\ln (s_k-s_k^\delta)
<\frac{1+\varepsilon}{1-\varepsilon}\ln n.
\end{eqnarray*}
In other words, for $\pi^{-1}(\omega)\in(s_k-s_k^\delta,s_k](k\geqslant N)$, we have $\eqref{eq33}.$
\end{proof}

%%In what follows, we  fix $m\geqslant1.$  We  write $X=(\mathbb{N}^m)^\mathbb{N}$, i.e. for every $x=(x_1\cdots x_n\cdots)\in X$, we  write $x=(x^{(1)}\cdots x^{(k)}\cdots)\in(\mathbb{N}^m)^\mathbb{N}$, where $x^{(k)}:=x|_{(k-1)m+1}^{km}$ is called the $k$-th {{\em segment}} of $x$. We call $x^{(k)}$
 %%{{\em square-segment}} (we write {\bf{s.s.}} for abbreviation) if the interval $(k-1)m+[1,m]$ contains a square of a positive integer $r$. At the same time we call $x^{(k)}$ $r$-th {\bf{s.s.}}. Since the set of squares is of zero density, there is few square-segments.
 For any $0<\varepsilon,\delta<1,$ let $N$ and $\{s_k\}_{k\geqslant1}$
  be the same as in Lemma \ref{m-con-exp}. By Proposition \ref{seed} and $\eqref{eq31}$, we can choose a  seed
$z=(z_n)_{n\geqslant1}\in G_{\mu}$ such that
\begin{equation}\label{eq36}
{-}\sum_{k=1}^{(n+1)^2}\ln{\nu([z_{k}])}\ll \sum_{k=1}^n{\ln{s_k}}.
\end{equation}
We  use this seed to generate a  Cantor subset in $G_\mu$ large enough. Roughly speaking, we replace the word $z_{k^2}$  by any word in $\pi((s_k-s_k^\delta,s_k])$. More precisely, define

$$F_z(\varepsilon,\delta)=\Bigg\{x\in X:x_n\Big\{\begin{array}{lc}
\in\pi((s_k-s_k^\delta,s_k]),~\textmd{if}~ n=k^2 ~\textmd{and}~k>N,\\
=z_n,~\textmd{otherwise}.\end{array}\Bigg\}$$

Since our modification of $z$ is made on square integer coordinates  which are of zero density,  it is easy to check  $F_z(\varepsilon,\delta)\subset G_\mu.$
The following proposition immediately implies that $\dim_\nu G_\mu\geqslant \alpha_{\nu}.$
\begin{proposition}\label{billingsley}
For any $0<\varepsilon<1$ and $0<\delta<1$, we have
$$\dim_\nu F_z(\varepsilon,\delta)\geqslant \frac{\alpha_{\nu}(1-\varepsilon)\delta}{1+\varepsilon}.$$
\end{proposition}
\begin{proof}
In order to apply  Billingsley's theorem, we are going to construct a measure $\lambda$ supported by $F_z(\varepsilon,\delta)$ such that for any $x\in F_z(\varepsilon,\delta)$,
$$\liminf_{n\rightarrow\infty}\frac{\ln{\lambda([x|^n_1])}}{\ln\nu([x|^n_1])}\geqslant \frac{\alpha_{\nu}(1-\varepsilon)\delta}{1+\varepsilon}.$$
For $x\in F_z(\varepsilon,\delta)$ and for $k^2\leq n<(k+1)^2 (k\geq N)$ define

 $$\lambda([x|_1^n])=\prod_{r=N}^k\frac{1}{s^\delta_r}.$$
The measure $\lambda$ is well defined on $F_z(\varepsilon,\delta)$.

%% For $n\geqslant N^2$, there exists a unique integer $t\geqslant N$ such that
 %%$t^2\leqslant n<(t+1)^2$. A euclidean division of $n$ by $m$ gives $n=qm+r$ where
  %%$q$ and $r$ are two integers and $0\leqslant r<q$. Also, there exist two integers $q_t$ and $r_t(0\leqslant r_t< q_t)$ such that $t^2=q_tm+r_t$.
For any $x\in F_z(\varepsilon,\delta)$, by the quasi Bernoulli property of $\nu$ we have
  %\begin{eqnarray*}
$$\nu([x|^n_1])\geqslant\frac{1}{C^{n+1}}\prod_{k=1}^{n}\nu([x_{k}])=
\frac{1}{C^{n+1}}\prod_{m=1}^{*n}{\nu([z_{m}])}\prod_{r=N}^{k}{\nu([x_{r^2}])}$$
%\end{eqnarray*}
where $*$ signifies the absence of the square numbers in $[N,n]$ in the product.
In combination  with the definition of $\lambda$, this yields
% \begin{equation}
$$\frac{\ln \lambda([x|^n_1])}{\ln\nu([x|^n_1])}\geqslant\frac{\delta\sum\limits_{r=N}^{k}{\ln {s_r}}}{(n+1)\ln C-\sum\limits_{m=1}^{n}{\ln \nu([z_{m}])}-\sum\limits_{r=N}^{k}{\ln{\nu([x_{r^2}])}}}.$$
%\end{equation}

Since for $x\in F_z(\epsilon,\delta), ~{x_{r^2}}\in\pi((s_r-s_r^\delta,s_r]),$ we get
 $$\pi^{-1}(x_{r^2})\in(s_r-s_r^\delta,s_r].$$ Thus by $\eqref{eq33}$, we have

$$-\ln\nu([x_{r^2}])<\frac{1}{\alpha_{\nu}}\frac{1+\epsilon}{1-\epsilon}\ln s_r.$$
Therefore,  by the above  inequality we have
%\begin{equation}\label{123}
$$\frac{\ln \lambda([x|^n_1])}{\ln\nu([x|^n_1])}>\frac{\delta\sum\limits_{r=N}^{k}{\ln {s_r}}}{(n+1)\ln C-\sum\limits_{m=1}^{(k+1)^2}{\ln \nu([z_{m}])}
+\frac{1}{\alpha_{\nu}}\frac{1+\epsilon}{1-\epsilon}\sum\limits_{r=N}^{k}{\ln s_r}}.$$
%\end{equation}

%Letting $n\rightarrow\infty$ and by

Letting $n\rightarrow\infty$ and  by $\eqref{eq36}$, we obtain, for any $x\in F_z(\epsilon,\delta),$
$$\liminf_{n\rightarrow\infty}{\frac{\ln \lambda([x|^n_1])}{\ln\nu([x|^n_1])}}\geqslant\frac{\alpha_{\nu}(1-\varepsilon)\delta}{1+\varepsilon}.$$

%In the second case, a similar argument will produce the same conclusion.
\end{proof}

Next, we  determine the other lower bound.

 \begin{proposition}\label{beta-lowerbound}
Let $\mu\in\mathcal{M}(X,T)$ be an invariant Borel probability measure.
 Assume that $\varphi$ has summable variations and admits a unique Gibbs measure $\nu$. Then  we have
$$\dim_\nu G_\mu\geqslant\beta(\nu|\mu).$$
%where$$\beta=\limsup_{k\rightarrow\infty}\limsup_{N\rightarrow\infty}\frac{-\sum\limits_{\omega\in\Sigma^k_N}\mu[\omega]\ln\mu[\omega]}
                                                      % {-\sum\limits_{\omega\in\Sigma^k_N}\mu[\omega]\ln\nu[\omega]}.  $$
\end{proposition}
\begin{proof}
If $h(\nu|\mu)=+\infty$, by Proposition  \ref{alpha-beta} we have
$\beta(\nu|\mu)\le \alpha_\nu$. Then there is nothing to prove because of Proposition \ref{alpha-upperbound}. It remains to consider the case of $h(\nu|\mu)<+\infty.$
Recall that we just need to consider the case of $\mu\neq\nu$. By  Proposition \ref{alpha-beta}, what we have to prove is $\dim_\nu G_\mu\geqslant\frac{h_\mu}{h(\nu|\mu)}$.

We are going to construct a subset $Y^*\subset G_\mu$ and a probability measure $\mu^*$ which have positive
mass on $Y^*$. Then we will apply the Billingsley theorem.

The symbols $\mu_j$, $m_j,~b_j$ and $M_j$  in what follows come from Proposition \ref{app}.
  Let
%$$H_j:=\{(\omega_1\cdots\omega_{m_j})\in\mathbb{N}^{m_j}~ | ~\omega\in G'_j\},$$
$$Y_j:=\{z|_1^{n_j}\in \mathbb{N}^{n_j}:z\in M_j\},$$
where $n_j$ is recursively defined as follows
$$n_j\geqslant\max\{m_{j+1}b_{j+1},N^2_{j-1}\},~N_0=0, ~N_j=\sum_{k=1}^jn_k,~ \forall j\geqslant1.$$
Define
$$Y^*:=\prod_{j=1}^\infty Y_j.$$
By a similar argument as in the proof of Proposition \ref{seed} one can show $Y^*\subset G_\mu.$
Now let us  construct a measure $\mu^*$.
% Let $\widetilde{\mu}_j$ be the probability vector on $\mathbb{N}^{n_j}$ such that
% $\widetilde{\mu}_j(w)=\mu_j([w])$ for $w\in \mathbb{N}^{n_j}$.
  For $x\in X$ and $N_j< n\leqslant N_{j+1}$, define
%\begin{equation}\label{constructionmeasure}
$$\mu^*([x|^n_1])=\prod_{k=1}^{j}\mu_{k}([x|_{N_{k-1}+1}^{N_{k}}])\times\mu_{j+1}([x|_{N_j+1}^n]).$$
%\end{equation}
This defines a probability measure on $X$.%, which will be simply denoted by   $$\mu^*:=\prod_{j=1}^\infty\widetilde{\mu}_j.$$

Since $\mu_j(M_j)>1-\frac{1}{4^j}$, a simple computation shows $\mu^*(Y^*)>\frac{2}{3}$.
For any $x\in Y^*$, by the definition of $\mu^*$ and the Gibbsian property of $\nu$  we have
\begin{eqnarray*}
\frac{\ln\mu^*([x|^n_1])}{\ln\nu([x|^n_1])}
%&\geqslant&\frac{-\ln\prod_{k=1}^{j-1}\widetilde{\mu}_k[x|_{T_k+1}^{T_{k+1}}]\times\widetilde{\mu}_{j+1}[x|_{T_j+1}^n]}{\ln %C-\sum_{k=0}^{n-1}\varphi(T^kx)}\\
&\geqslant&\frac{-\sum_{k=1}^{j}\ln\mu_{k}([x|_{N_{k-1}+1}^{N_{k}}])-
\ln\mu_{j+1}([x|_{N_j+1}^n])}{\ln C-\sum_{k=0}^{N_j-1}\varphi(T^kx)-\sum_{k=N_j}^{n-1}\varphi(T^kx)}.\\
\end{eqnarray*}
Here we need to deal with two cases separately.
If $N_j< n<N_j+m_{j+1}$, % for any $x\in [\omega_1\cdots\omega_n]$ there exists $y\in M_{j+1}$ such that  $y|_1^{m_{j+1}}=T^{T_j}(x)|_1^{m_{j+1}}$, by the summable variation of $\varphi$,
   we have
\begin{eqnarray*}
\frac{\ln\mu^*([x|^n_1])}{\ln\nu([x|^n_1])}
%&\geqslant&\frac{-\ln\prod_{k=1}^{j-1}\widetilde{\mu}_k([\omega^{(k)}])-\ln\widetilde{\mu}_{j+1}([\omega_{T_j+1}\cdots\omega_n])}{\ln %C-\{\sum\limits_{k=0}^{T_j-1}\varphi(T^k\omega)+\sum\limits_{k=T_j}^{n-1}\varphi(T^k\omega)\}}\\
&\geqslant&\frac{-\sum_{k=1}^{j}\ln\mu_{k}([x|_{N_{k-1}+1}^{N_{k}}])}{\ln C-\sum_{k=0}^{N_j-1}\varphi(T^kx)-\sum_{k=N_j}^{n-1}\varphi(T^kx)}.
\end{eqnarray*}
For $1\leqslant k\leqslant j$, by the definition of $Y^*$ we have
$$T^{N_k}(x)\in M_{k+1}.$$
Hence, by   Proposition \ref{app} we have
             $$-\ln\mu_{k}([x|_{N_{k-1}+1}^{N_{k}}])\geqslant n_k(h_{\mu_k}-\frac{1}{4^k}),\,~-\sum_{k=T_j}^{n-1}\varphi(T^kx)\leqslant b_{j+1},$$
and
$$-\sum\limits_{k=0}^{N_j-1}\varphi(T^kx)\leqslant\sum\limits_{k=1}^{j}{n_k(h(\nu|\mu_k)-\frac{1}{4^k})}.$$
Thus we have
\begin{eqnarray*}
\frac{\ln\mu^*([x|^n_1])}{\ln\nu([x|^n_1])}
&\geqslant&\frac{\frac{1}{n}\sum_{k=1}^{j}{n_k(h_{\mu_k}-\frac{1}{4^k})}}{\frac{1}{n}\ln C
+\frac{1}{n}\sum_{k=1}^{j}{n_k(h(\nu|\mu_k)-\frac{1}{4^k})}+\frac{1}{n}b_{j+1}}.
\end{eqnarray*}
By Proposition \ref{markov-app} and  \ref{int-app}, we have $h_{\mu_k}\to h_\mu$ and $h(\nu|\mu_k)\to h(\nu|\mu)$ as $k\to\infty$.
Noting that  $$\frac{n-N_j}{n}< \frac{m_{j+1}}{n}\leqslant\frac{1}{b_{j+1}}\to0,\,\,\,\frac{ b_{j+1}}{n}<\frac{b_{j+1}}{n_{j}}\leqslant\frac{1}{m_{j+1}}\to0~\textmd{as}~ n\to\infty,$$ %by Ces\`aro summation
 we have
$$\liminf_{n\to\infty}\frac{\ln\mu^*([x|^n_1])}{\ln\nu([x|^n_1])}\geqslant\frac{h_\mu}{h(\nu|\mu)}.$$
If $N_j+m_{j+1}\leqslant n\leqslant N_{j+1}$, by Proposition \ref{app} we have
$$\sum_{k=T_j}^{n-1}\varphi(T^kx)\leqslant (n-N_j)(h(\nu|\mu_{j+1})-\frac{1}{4^{j+1}}).$$
Thus
\begin{eqnarray*}
\frac{\ln\mu^*[x|^n_1]}{\ln\nu([x|^n_1])}
&\geqslant&\frac{\frac{1}{n}\sum\limits_{k=1}^j{n_k(h_{\mu_k}-\frac{1}{4^k})}+\frac{n-N_j}{n}(h_{\mu_{j+1}}
-\frac{1}{4^{j+1}})}{\frac{1}{n}\ln C+\frac{1}{n}\sum\limits_{k=1}^j{n_k(h(\nu|\mu_k)-\frac{1}{4^k})}+\frac{n-N_j}{n}(h(\nu|\mu_{j+1})-\frac{1}{4^{j+1}})}\\
&\longrightarrow&\frac{h_\mu}{h(\nu|\mu)}~\textmd{as}~ n\to\infty.
\end{eqnarray*}
That means for any $x\in Y^*$  we have
 $$\liminf_{n\to\infty}\frac{\ln\mu^*([x|^n_1])}{\ln\nu([x|^n_1])}\geqslant\frac{h_\mu}{h(\nu|\mu)}.$$
 By Billingsley's theorem, we have $\dim_\nu Y^*\geqslant\frac{h_\mu}{h(\nu|\mu)}$.  Thus the proof is completed.
   \end{proof}

 By Propositions \ref{alpha-upperbound} and \ref{beta-lowerbound} we have the following lower bound.

\begin{theorem}  \label{lowerbound}
 Let $\mu\in\mathcal{M}(X,T)$ be an invariant Borel probability measure.
 Assume that $\varphi$ has summable variations and admits a unique Gibbs measure $\nu$ with convergence exponent $\alpha_\nu$. Then  we have $$\dim_\nu G_\mu\geqslant\max\{\alpha_\nu,\beta(\nu|\mu)\}.$$
 \end{theorem}

\section{The upper bound of the Hausdorff dimension of $G_\mu$}
In the section, we will prove the upper bound of $\dim_\nu G_\mu$. By Lemma \ref{cr}, $G_\mu$ is the set of those points $x$ such that
$$\forall ~[u]\in\mathcal{C}^*,~\lim_{n\to\infty}\Delta_{x,n}([u])=\mu([u]).$$
Thus, for any two fixed integers $N\geqslant1$ and $j\geqslant1$, we have
$$G_\mu\subset\left\{x\in X:\forall u\in\Sigma_N^j,~\lim_{n\to\infty}\Delta_{x,n}([u])=\mu([u])\right\}.$$
Note that $\Delta_{x,n}([u])$ is the frequency of appearance of $u$ in the word $x|_1^{n+|u|-1}$.
For every word ~$\omega\in\Sigma_N^n$ ~of length ~$n$~and every word~$u\in\Sigma_N^k$~of length ~$k$ with $k\leqslant n$, denote by $p(u|\omega)$ the frequency of appearances of $u$ in $\omega$, i.e.,
$$p(u|\omega)=\frac{\tau_u(\omega)}{n-k+1},$$
where $\tau_u(\omega)$ denotes the number of $j$ with $1\leqslant j\leqslant n-k+1,$  so that $\omega_j\cdots\omega_{j+k-1}=u.$
We need  a  combinatorial lemma.   %Set $$H_k(\omega):=\sum_{u\in\Sigma_N^k}-p(u|\omega)\ln p(u|\omega).$$
\begin{lemma}[\cite{GW}]\label{entropy-comb}
For any $h>0,~\delta>0,~k\in\mathbb{N}$ and  $n\in\mathbb{N}$ large enough, we have
$$\sharp\left\{\omega\in\Sigma_N^n:\sum_{u\in\Sigma_N^k}-p(u|\omega)\ln p(u|\omega)\leqslant kh\right\}\leqslant\exp(n(h+\delta)).$$
\end{lemma}

\begin{theorem}  \label{upperbound}
Let $\mu\in\mathcal{M}(X,T)$. Assume that $\varphi$ has summable variations and admits a unique Gibbs measure $\nu$ with convergence exponent $\alpha_\nu$. Then  we have
  $$\dim_\nu G_\mu\leqslant\max\{\alpha_\nu,\beta(\nu|\mu)\}.$$
\end{theorem}
\begin{proof}

We fix two integers $N\geqslant1$ and $j\geqslant1$, which first $N$ then $j$ will tend to the infinity. According to the above analysis,  we have
$$G_\mu\subset\left\{x\in X:\forall u\in \Sigma_N^j,~\lim_{n\to\infty}\Delta_{x,n}([u])=\mu([u])\right\}.$$
Note that $\Delta_{x,n}([u])=n^{-1}\tau_u(x|_{1}^{n+j-1})$. It follows that for any $\epsilon>0,$ we have

$$G_\mu\subset \bigcup_{l=1}^\infty\bigcap_{n=l}^\infty H_n(\epsilon,j,N),$$
where $$H_n(\epsilon,j,N):=\left\{x\in X:\forall u\in\Sigma^j_N,~\left|n^{-1}\tau_u(x|_{1}^{n+j-1})-\mu([u])\right|<\epsilon~\right\}.$$
 So, by the $\sigma$-stability of the Hausdorff dimension we have
 %\begin{equation}\label{dimensionstability}
 $$\dim_\nu G_\mu\leqslant \sup_{l\geqslant1}\dim_\nu\bigcap_{n=l}^\infty H_n(\epsilon,j,N).$$
% \end{equation}
 In order to estimate the dimension of $\bigcap_{n=l}^\infty H_n(\epsilon,j,N)$, let us consider the $(n+j-1)$-prefixes of the points in $H_n(\epsilon,j,N)$:

 $$\Lambda_n(\epsilon,j,N):=\Big\{x_1\cdots x_{n+j-1}\in\mathbb{N}^{n+j-1}:x\in H_n(\epsilon,j,N)\Big\}.$$
Let $\delta_n=\sup\{\nu([\omega]):\omega\in\mathbb{N}^n\}$. We have $\lim_{n\to\infty}\delta_n=0$ by Proposition \ref{metricuniform}. Then the cylinder set $\{[\omega]:\omega\in\Lambda_n(\epsilon,j,N)\}$ forms a $\delta_{n+j-1}$-covering of  $\bigcap_{n=l}^\infty H_n(\epsilon,j,N)$.
 Assume that $\gamma>\max\{\alpha_\nu,\beta(\nu|\mu)\}$ and $\gamma<3/2$ without loss of generality. By the definition of  $\gamma$- Hausdorff measure
$$\mathcal{H}^\gamma(\bigcap_{n=l}^\infty H_n(\epsilon,j,N))=\lim_{n\to\infty}\mathcal{H}_{\delta_{n+j-1}}^\gamma(\bigcap_{n=l}^\infty H_n(\epsilon,j,N))\leqslant\liminf_{n\to\infty}\sum_{\omega\in\Lambda_n(\epsilon,j,N)}\nu([\omega])^\gamma.$$
Given a word $\omega\in\Lambda_n(\epsilon,j,N)$,  we consider $(\tau_u(\omega))_{u\in\Sigma_N^j}$, which would be  called the appearance distribution with respect to $\Sigma_N^j$  of  $\omega$. Denote by $\mathfrak{D}_n(\epsilon,j,N)$ the set of such appearance distributions of all elements of $\Lambda_n(\epsilon,j,N)$. Given a distribution $(\tau_u)\in \mathfrak{D}_n(\epsilon,j,N)$, set
$$A((\tau_u)):=\{\omega\in\Lambda_n(\epsilon,j,N):\tau_u(\omega)=\tau_u,\,\forall u\in\Sigma_N^j\}.$$
Then $\Lambda_n(\epsilon,j,N)$ is partitioned  into $A((\tau_u))$'s. % so that we have
%$$\sum_{\omega\in\Lambda_n(\epsilon,j,N)}\nu([\omega])^\gamma=\sum_{(\tau_u)\in D_n(\epsilon,j,N)}\,\sum_{\omega\in A((\tau_u))}\nu([\omega])^\gamma.$$
Note that there are ${N^j}$ possible words  $u$ in $\Sigma_N^j$ and that $n(\mu([u])-\epsilon)\leqslant\tau_u\leqslant n(\mu([u])+\epsilon)$, i.e.  $\tau_u$ varies in an interval of length  $2\epsilon n$.  It follows that
$\sharp \mathfrak{D}_n(\epsilon,j,N)\leqslant (2\epsilon n)^{N^j}$. This, together with the above partition, leads to
 \begin{equation}\label{uppercovering}
 \sum_{\omega\in\Lambda_n(\epsilon,j,N)}\nu([\omega])^\gamma\leqslant(2\epsilon n)^{N^j}\max_{(\tau_u)\in \mathfrak{D}_n(\epsilon,j,N)}\sum_{\omega\in A((\tau_u))}\nu([\omega])^\gamma.
 \end{equation}

 Our task is to estimate the sum on the right-hand side in the above inequality. We first decompose $A((\tau_u))$ into disjoint union of some sets. %Next we are going to  estimate the sum in the above inequality by using Lemma \ref{entropy-comb}.
 Given $\omega=\omega_1\cdots\omega_{n+j-1}\in A((\tau_u)),$ we say $\omega_k\omega_{k+1}\cdots\omega_{k+m-1}$ is a {\bf{maximal $(N,j)$-run subword}} of $\omega$
 if the following conditions are satisfied
   \begin{enumerate}
\item[(1)]
$m\geqslant j$,
\item[(2)]
$\forall 0\leqslant i\leqslant m-1,~\omega_{k+i}\leqslant N\,\textmd{and} \,\,\omega_{k-1}>N,\, \omega_{k+m}>N.$
\end{enumerate}
  On the other hand, a subword between {\bf{maximal $(N,j)$-run subword}}s is called ``bad subword".
  The set $A((\tau_u))$ is just a collection of  words like
   \begin{equation}\label{length-dis}
   \omega=B_{r_1}W_{n_1}B_{r_2}\cdots W_{n_t}B_{r_{t+1}},
   \end{equation}
  where $B_{r_i}$ denotes ``bad subword" with length $r_i$ and $W_{n_i}$ denotes {\bf{maximal $(N,j)$-run subword}} with length $n_i$. Write
  $${K}:=\sum_{u\in\Sigma_N^j}\tau_u ~\text{and}~
  ~s:=\left\lfloor\frac{n-{K}}{j}\right\rfloor+1.$$
  It is easily seen that $t\leqslant s$. In other word, every element in $A((\tau_u))$ has at most $s$ {\bf{maximal $(N,j)$-run subword}}s. Furthermore, by writing  ${K}_t:=\sum_{i=1}^tn_i,$ we have $$K_t={K}+t(j-1)$$ and
  \begin{equation}\label{bad subword number}
  r_1\geqslant0,r_{t+1}\geqslant0, r_i\geqslant 1~(2\leqslant i\leqslant t)~\text{and}\,\sum_{i=1}^{t+1}r_i=n+j-1-K_t.
  \end{equation}
   For $1\leqslant t\leqslant s$, we denote by $A_t$  the set of  words  in $A((\tau_u))$ with  $t$ {\bf{maximal $(N,j)$-run subword}}s.
It is clear that  $A((\tau_u))$ is partitioned into $A_t$'s, i.e.
 \begin{equation}\label{1-partition}
  A((\tau_u))=\bigsqcup_{t=1}^sA_t.
  \end{equation}
% Then we have
 % \begin{equation}\label{atpartition}
 % \sum_{\omega\in A((\tau_u))}\nu([\omega])^\gamma=\sum_{t=1}^{s}\sum_{\omega\in A_t}\nu([\omega])^\gamma.
 % \end{equation}
 Next, we partition $A_t$ by the length pattern of ``bad subword" and {\bf{maximal $(N,j)$-run subword}}. Recall that every word $\omega\in A_t$ has the form \eqref{length-dis}. We call $(r_1,n_1,r_2,\cdots,n_t,r_{t+1})$ the length pattern of ``bad subword" and {\bf{maximal $(N,j)$-run subword}}. Denote by $\mathfrak{L}_t$ the set of all such length pattern of $\omega$ in $A_t$. Given a length pattern $({\bf{r,n}}):=(r_1,n_1,\cdots,n_t,r_{t+1})\in \mathfrak{L}_t$, let $B({\bf r,n})$  denote the set of elements of $A_t$ with the length pattern $({\bf r,n})$. %By paying attention to  \eqref{bad subword number} and the fact that $n_i\geqslant j$ for $1\leqslant i\leqslant t$, one can give the following estimate $$\sharp \mathfrak{L}_t\leqslant \frac{(K-1)!}{(K-t)!(t-1)!}\frac{(n-K-(t-1)j+t)!}{t!(n-K-(t-1)j)!}.$$
  Thus, $A_t$ is partitioned into  $B({\bf r,n})$'s.  It follows that

 \begin{equation}\label{length-partition}
   \sum_{\omega\in A_t}\nu([\omega])^\gamma\leqslant \sharp \mathfrak{L}_t\max_{({\bf r,n})\in\mathfrak{L}_t}\sum_{\omega\in B({\bf r,n})}\nu([\omega])^\gamma.
   \end{equation}
  %Now we  estimate the sum over $A_t$ by deleting ``bad subwords''.
   Let ${A}'_t$ be the set of finite words by deleting all ``bad subwords'' of $\omega$ in $A_t$. Thus by the quasi Bernoulli property, we have
\begin{eqnarray*}
   \sum_{\omega\in B({\bf r,n})}\nu([\omega])^\gamma
   %&\leqslant&C^\gamma\sum_{\omega\in B({\bf r,n})}\exp(\gamma\sum_{i=0}^{n+j-2}\varphi(T^ix_\omega))\\
   &\leqslant&C^{2\gamma(t+1)}\sum_{\omega\in B({\bf r,n})}\prod_{i=1}^{t+1}\nu([B_{r_i}(\omega)])^\gamma\prod_{i=1}^t\nu([W_{n_i}(\omega)])^\gamma\\
  &\leqslant&C^{2\gamma(t+1)}\sum_{\omega\in B({\bf r,n})}\prod_{i=1}^{t+1}\nu([B_{r_i}(\omega)])^\gamma\sum_{\omega\in B({\bf r,n})}\prod_{i=1}^t\nu([W_{n_i}(\omega)])^\gamma\\
   &\leqslant&C^{\gamma(4t+5)} V \sum_{\omega\in A'_t}\nu([\omega])^\gamma,\\
   \end{eqnarray*}
   where
   $$V:=\sum_{\omega\in B({\bf r,n})}\nu([B_{r_1}(\omega)\cdots B_{r_{t+1}}(\omega)])^\gamma.$$
  According to \eqref{bad subword number}, we have $\sum_{i=1}^{t+1}r_i\leqslant n-K$, Thus, together with  Lemma \ref{exponent-bound}, this yields
  $$V\leqslant \sum_{\omega\in\mathbb{N}^{n-K}}\nu([\omega])^\gamma\leqslant C_0M_0^{n-K}.$$
     Then, by \eqref{length-partition}  we have
   \begin{equation}\label{distributionpartition}
   \sum_{\omega\in A_t}\nu([\omega])^\gamma\leqslant C_0C^{\gamma(4t+5)}M_0^{n-K}\sharp \mathfrak{L}_t \sum_{\omega\in A'_t}\nu([\omega])^\gamma .
   \end{equation}
From the definition of $H_n(\epsilon,j,N)$, we have $$\sum_{u\in\Sigma_N^j}(\mu([u])-\epsilon)<\frac{K}{n}<\sum_{u\in\Sigma_N^j}(\mu([u])+\epsilon).$$
 For any $\delta'>0$, one can choose $N$ large enough and $\epsilon$ small enough i.e. $\epsilon=N^{-2j}$ such that
 $$1-\sum_{u\in\Sigma_N^j}\mu([u])<\delta'~\text{and} ~\epsilon N^j<\delta'. $$
It follows that
\begin{equation}\label{nk}
1-\frac{K}{n}<\epsilon N^j+\delta'<2\delta' .
\end{equation}
  According to the definition of $s$ and the fact that  $t\leqslant s$, for any $\delta>0$, when $\delta'$ is taken small enough we have
  \begin{equation}\label{uppercinequ}
  C_0C^{\gamma(4t+5)}M_0^{n-K}\leqslant e^{n\delta/2}.
  \end{equation}
On the other hand, we observe that very length pattern $({\bf{r,n}})\in \mathfrak{L}_t$ is just corresponding to the integer solution of the following equation set
$$\Big\{\begin{array}{lc}
\sum_{i=1}^tn_i=K_t, ~n_i\geqslant j(1\leqslant i\leqslant t),\\
\sum_{i=1}^{t+1}r_i=n+j-1-K_t,~r_1\geqslant0,~r_{t+1}\geqslant0,~ r_i\geqslant 1~(2\leqslant i\leqslant t).\end{array}$$
By the element combinatorial theory, one can obtain the following estimate $$\sharp \mathfrak{L}_t\leqslant \frac{(K-1)!}{(K-t)!(t-1)!}\frac{(n-K-(t-1)j+t)!}{t!(n-K-(t-1)j)!}.$$ Noting that $j$ is a fixed integer relative to $n$, by the Stirling formula we have
%\begin{equation}\label{Q-bound}
$$\frac{(K-1)!}{(K-t)!(t-1)!}\frac{(n-K-(t-1)j+t)!}{t!(n-K-(t-1)j)!}\leqslant e^{n\delta/2}.$$
%\end{equation}
 In combination with \eqref{distributionpartition} and  \eqref{uppercinequ}, this yields
\begin{equation}\label{upperaaineq}\sum_{\omega\in A_t}\nu([\omega])^\gamma\leqslant e^{n\delta}\sum_{\omega\in A'_t}\nu([\omega])^\gamma.
\end{equation}

Now we  estimate the sum on the right-hand side in the above inequality. First we consider a set
$$
\widetilde{A}_t=\{\widetilde{B}_{r_1}W_{n_1}\widetilde{B}_{r_2}\cdots W_{n_t}\widetilde{B}_{r_{t+1}}:\omega=B_{r_1}W_{n_1}B_{r_2}\cdots W_{n_t}B_{r_{t+1}}\in A_t\},
$$
where $\widetilde{B}_{r_i}$ is a finite word composed of digit $N+1$ with length $r_i$. In other word, the set $\widetilde{A}_t$ is just a set of finite words obtained  by replacing each ``bad subword" $B_{r_i}$ of $\omega$ in $A_t$ by a finite word composed of digit $N+1$ with length $r_i$. Thus the two sets $\widetilde{A}_t$ and $A'_t$ have the same cardinal and each subword $u\in \Sigma_N^j$ appears $\tau_u$ times in  $\omega$ of $\widetilde{A}_t$.  Take

$$h=\frac{1}{j}(\sum_{u\in\Sigma_N^j}-\frac{\tau_u}{n}\ln\frac{\tau_u}{n}-\frac{n-K}{n}\ln \frac{n-K}{n})$$in Lemma \ref{entropy-comb}. Then, for the same $\delta>0$ as  above and for $n$ large enough we have
\begin{eqnarray}\label{estimation}
\sharp A'_t=\sharp \widetilde{A}_t\leqslant&\exp(n(h+\delta)).
\end{eqnarray}
Given $\omega\in A'_t$, denote by $(\tau'_u)$ the appearance distribution with respect to $\Sigma_N^j$ of $\omega$. Then, we have
\begin{equation}\label{number-bound}
     |\omega|=K_t\,\,\textmd{and}\,\,\tau_u \leqslant\tau'_u\leqslant \tau_u+(t-1)(j-1)\leqslant \tau_u+n-{K}.
   \end{equation}
   By the Gibbsian property  and \eqref{p-bound}, we have

\begin{eqnarray*}
j\ln\nu([\omega_1\cdots \omega_{K_t}])&\leqslant&
j\ln C+j\sum_{i=0}^{K_t-1}\varphi(T^ix)\\
&=&j\ln C+\sum_{i=0}^{j-2}(j-1-i)\varphi(T^ix)\\
&&+\sum_{i=K_t-j+1}^{K_t-1}(K_t-i)\varphi(T^ix)+\sum_{i=0}^{K_t-j}\sum_{k=0}^{j-1}\varphi(T^{i+k}x)\\
&\leqslant&(K_t+j^2-j+1)\ln C+\sum_{u\in\Sigma_N^j}\tau'_u\ln\nu([u]),
\end{eqnarray*}
where we obtain the equality by  taking the sums along the oblique diagonals as shown in the following figure and  the last inequality follows from the Gibbsian property and (\ref{p-bound}).
\begin{center}
\setlength{\unitlength}{1mm}
\begin{picture}(45,24)
%\put(0,0){\vector(1,0){60}}
\multiput(0,16)(4,0){5}{\circle*{0.8}}
\multiput(24,16)(4,0){5}{\circle*{0.8}}
\multiput(19,16)(1,0){3}{\circle*{0.5}}

\multiput(0,12)(4,0){5}{\circle*{0.8}}
\multiput(24,12)(4,0){5}{\circle*{0.8}}
\multiput(19,12)(1,0){3}{\circle*{0.5}}

\multiput(0,5)(0,1){3}{\circle*{0.5}}
\multiput(0,0)(4,0){5}{\circle*{0.8}}
\multiput(24,0)(4,0){5}{\circle*{0.8}}
\multiput(19,0)(1,0){3}{\circle*{0.5}}

%\put(5,7){\line(-2,5){20}}
%\multiput(0,26)(3,-3){10}{\line(1,-1){3}}
\put(0,16){\line(1,-1){16}}
\put(40,0){\line(-1,1){16}}
\put(15,20){$K_t$}
\put(-5,6){$k$}

\end{picture}
\end{center}
Together with \eqref{estimation}, this yields
\begin{eqnarray*}%\label{estimation A'}
\sum_{\omega\in A'_t}\nu^\gamma([\omega])&\leq& \sharp A'_t \max_{\omega\in A'_t}\nu^\gamma([\omega])\\&\leq& \exp\left\{n(h+\delta)+\frac{\gamma}{j}\bigg\{\sum_{u\in\Sigma_N^j}\tau'_u\ln \nu([u])+(K_t+j^2)\ln C\bigg\}\right\}
\end{eqnarray*}
Rewrite the right-hand side of the above inequality as
$$\exp\{nL(\gamma,j,(\tau'_u))\},$$
where$$ L(\gamma,j,(\tau'_u))=h+\frac{\gamma}{j}\sum_{u\in\Sigma_N^j}\frac{\tau'_u}{n}\ln\nu([u])+\frac{\gamma}{jn}(K_t+j^2)\ln C+\delta.$$
Now we shall give a negative upper-bound of $ L(\gamma,j,(\tau'_u))$.
In virtue of the definition of $H_n(\epsilon,j,N)$ and \eqref{number-bound}, we can take $\epsilon>0$ small enough and $n$ large enough such that
\begin{equation}\label{estimation 2}\frac{1}{j}\sum_{u\in\Sigma_N^j}\frac{\tau'_u}{n}\ln \nu([u])\leq -\frac{1}{j}H_{j,N}(\nu,\mu)+\delta.\end{equation}
At the same time, we can take  $j$  large enough such that
$$\frac{1}{jn}(K_t+j^2)\ln C\leqslant\delta/2~~.$$
In combination with the last two inequalities, we have
\begin{eqnarray*}
L(\gamma,j,(\tau'_u))&\leq& h-\gamma\frac{1}{j}H_{j,N}(\nu,\mu)+(3/2+\gamma)\delta.
\end{eqnarray*}
Recall that
$$3/2>\gamma>\beta(\nu|\mu),$$
We take $\delta>0$ small enough and $j, \,N$ large enough such that
\begin{equation}  \label{gamma-ineq}
\gamma\geqslant\frac{\frac{1}{j}H_{j,N}(\mu,\mu)+7\delta}
{\frac{1}{j}H_{j,N}(\nu,\mu)}.
\end{equation}
By a similar argument of \eqref{estimation 2}, we have $$-\frac{1}{j}\sum_{u\in\Sigma_N^j}\frac{\tau_u}{n}\ln\frac{\tau_u}{n}\leq\frac{1}{j}H_{j,N}(\mu,\mu)+\delta.$$
It follows that
\begin{eqnarray*}
L(\gamma,j,(\tau'_u))&\leq& h-\frac{1}{j}H_{j,N}(\mu,\mu)-4\delta\\
&\leq& -\frac{1}{j}\frac{n-K}{n}\ln \frac{n-K}{n}-3\delta,
\end{eqnarray*}
 By \eqref{nk}, we can take $\delta'$ small enough such that
$$-\frac{n-K}{n}\ln \frac{n-K}{n}\leq \delta.$$
Thus, we have
$$L(\gamma,j,(\tau'_u))\leq -2\delta,$$
and
$$\sum_{\omega\in A'_t}\nu^\gamma([\omega])\leq \exp (-2n\delta).$$
In combination with \eqref{uppercovering}, \eqref{1-partition} and (\ref{upperaaineq}), this yields
$$\sum_{\omega\in \Lambda_n(\epsilon,j,N)}\nu([\omega])^\gamma\leqslant s(2\epsilon n )^{N^j} e^{-n\delta},$$
which implies that for any $\gamma>\max\{\alpha_\nu,\beta(\nu|\mu)\},$
$$\mathcal{H}^\gamma(\bigcap_{n=l}^\infty H_n(\epsilon,j,N))=\lim_{n\to\infty}\mathcal{H}_{\delta_{n+j-1}}^\gamma(\bigcap_{n=l}^\infty H_n(\epsilon,j,N))=0.$$
Then it follows that$$\dim_\nu G_\mu\leqslant\gamma.$$
       Thus, we obtain  $$\dim_\nu G_\mu\leqslant\max\{\alpha_\nu,\beta(\nu|\mu)\}.$$

\end{proof}

%\vspace{2\baselineskip}
  \section{Applications}
    In the section, we study the set of generic points of an invariant measure for an expanding interval map   by transferring dimension results from the symbolic    space to the interval $[0,1)$. However, for convenience of presentation, we choose to work with the continued fraction system. Then we will give the proof of Theorem \ref{cfs}. Furthermore, we will describe the Hausdorff dimension of the sets of generic points  of an invariant measure for the Gauss transformation  with respect to the Euclidean metric on $[0,1)$. The corresponding results will be stated without proof for a class of expanding interval maps.
\subsection{Gauss transformation}

    Define the Gauss transformation $S:[0,1)\to[0,1)$ by
    $$S(0):=0, \quad S(x):=\frac{1}{x}-\left\lfloor\frac{1}{x}\right\rfloor,\quad x\in(0,1).$$ Then every irrational number in $[0,1)$ can be written uniquely as an infinite expansion of the form
     $$x=\frac{1}{a_1(x)+\frac{1}{a_2(x)+\ddots}},$$
     where $a_1(x)=\lfloor\frac{1}{x}\rfloor$  and $a_k(x)=a_1(S^{k-1}(x))$  for $k\geqslant2$ are called the partial quotients of $x$. For simplicity
     we denote the expansion by $(a_1a_2\cdots)$. Let $\mathcal{M}([0,1), S)$ denote all invariant Borel probability measures with respect to $S$ on $[0,1)$. Set
    $$\Delta(a_1a_2\cdots a_n):=\{x\in[0,1):x_1=a_1,x_2=a_2,\cdots, x_n=a_n\}$$
    which is called a rank-$n$ basic interval.
   Define $\kappa:\mathbb{N}^{\mathbb{N}}\to(0,1)$ by
   $$\kappa(a_1a_2\cdots)=\frac{1}{a_1+\frac{1}{a_2+\ddots}}.$$
      Hence $$\kappa([a_1\cdots a_n])=\Delta(a_1\cdots a_n).$$
      This establishes a one-to-one correspondence between the cylinders in $X$ and the basic intervals in $[0,1)$.

\textit{Proof of Theorem \ref{cfs}.}
Recall that $\eta_s$ is the Gibbs measure associated to the potential $\phi_s(x)=2s\ln x$ for $s>\frac{1}{2}$.
% Denote by $\alpha_s$ the convergence exponent of $\eta_s$ i.e.
 %$$\alpha_s:=\inf\left\{t\geqslant0:\sum_{n=1}^\infty\eta_s(\Delta(n))^t<+\infty\right\}.$$
 %By the quasi Bernoulli property, we have  $$\alpha_1=s\alpha_s.$$
   %Note that $\eta_1$ is the Gauss measure with density $\frac{1}{\ln2}\frac{1}{ (1+x)}$. It follows that  $\alpha_1=\frac{1}{2}$. So  we obtain $$\alpha_s=\frac{1}{2s}.$$
   Define the potential function$$\varphi_s(\omega )=\phi_s\circ\kappa (\omega)$$ on $X$. Then $\varphi_s$ admits a unique Gibbs measure $\nu_s$  which has the same convergence exponent with $\eta_s$ and satisfies  for all cylinder $[\omega]\in\mathcal{C}^*$
  \begin{equation}\label{eta-equ-nu}
   \nu_s([\omega])=\eta_s(\Delta(\omega)).
   \end{equation}
          That implies that
      \begin{equation} \label{symbolic-r-e}
      \dim_{\nu_s}A=\dim_{\eta_s}\kappa(A)\end{equation} for any subset $A$ of $X$.

   Now define $$\theta:=\ell\circ \kappa$$
on $\mathcal{B}(X)$. Then the two systems
   $(X,\mathcal{B}(X),\theta,T)$ and $([0,1),\mathcal{B}([0,1)),\sigma,S)$ are isomorphic. Hence
   $h_\theta=h_{\ell}$.
     Let $G_\theta$ be the set of  generic points of $\theta$ on the system $(X,T)$. By a standard argument, one can cheeck $$\mathcal{G}_\ell=\kappa(G_\theta).$$
    According to the analysis following  Theorem  \ref{thm} and (\ref{symbolic-r-e}) we have completed the proof of Theorem \ref{cfs}.\quad\quad\quad\quad\quad\quad\quad\quad\quad\quad\quad\quad\quad\quad\quad\quad\quad\quad\quad\quad\quad\quad$\Box$

Denote by $\dim_H\mathcal{G}_{\ell}$ the Hausdorff dimension  of $\mathcal{G}_{\ell}$ with respect to the Euclidean metric on $[0,1)$. Our next result shows that $\dim_H\mathcal{G}_{\ell}$ equals the Billingsley dimension  $\dim_{\lambda}\mathcal{G}_{\ell}$, where $\lambda$ is the  Lebesgue measure.
First, we note a fact that for any subset $A\subset(0,1)$, we have $\dim_HA\leqslant\dim_\lambda A$, because the former dimension index is defined by using the covering of arbitrary intervals and the later one by the covering of basic intervals. By a result of Wegmann (\cite{weg}, see also \cite{cajar}, p. 360), the equality holds in the
    following situation.
    \begin{proposition} \label{carjar}
     One has $\dim_HA=\dim_\lambda A$ for $A\subset (0,1)$ if
     \begin{equation} \label{carjare}
     \lim_{n\to\infty}\frac{\ln\lambda(\Delta(x_1\cdots x_n))}{\ln\lambda(\Delta(x_1\cdots x_{n+1}))}=1, ~(\forall x \in A).
     \end{equation}

      \end{proposition}

 \begin{theorem}  \label{Leb-Haus}
             For any $\ell\in \mathcal{M}([0,1), S)$, we have
                  $$\dim_H\mathcal{G}_{\ell}=\max\left\{\frac{1}{2},
                  \frac{h_{\ell}}{-2\int_0^1\ln x\, \mathrm{d}\ell(x)}\right\}.$$
     \end{theorem}

\begin{proof}
Note that $\eta_1$ is the Gauss measure with density $\frac{1}{\ln2}\frac{1}{ (1+x)}$. It follows that  $\alpha_1=\frac{1}{2}$. Hence, $\nu_1$ also has convergence exponent $\frac{1}{2}$. By Theorem \ref{thm} and (\ref{symbolic-r-e}),
     we have
    $$\dim_H\mathcal{G}_{\ell}\leqslant \dim_\lambda\mathcal{G}_{\ell}
    =\dim_{\eta_1}\mathcal{G}_{\ell} =
    \dim_{\nu_1}G_\theta=\max\left\{\frac{1}{2},\frac{h_\theta}{-\int\varphi_1 \, \mathrm{d}\theta}\right\},$$
    where the first equality comes from the fact that the Gauss measure $\eta_1$ is boundedly equivalent to the Lebesgue measure $\lambda$.
   % =\max\{\alpha_{\widetilde{\nu}},\frac{h_{\widetilde{\lambda}}}{-\int\widetilde{\varphi}d\widetilde{\lambda}}\}

        It  remains to show the converse inequality.
          Recall that we have constructed two subsets for the lower bound estimation of $\dim_\nu G_\mu$ in Section 4 and we will use them once more. First,  by Proposition \ref{seed},
    there exists $z=(z_n)_{n\geqslant1}\in G_{\theta}$ such that
    \begin{equation}\label{gaussseed}
    -\sum_{k=1}^{(q+1)^2}\ln{\nu_1([z_k])}\ll \sum_{k=1}^qk^2.
    \end{equation}
   For a positive number $a>1$, set
      $$ F=\left\{x\in X:x_{k^2} \in(a^{k^2},2a^{k^2}];~x_k=z_k ~\textmd{if} ~k~ \textmd{is nonsquare}\right\}.$$
     It is clear that $F\subset  G_{\theta}$. Thus, we have
      $\dim_{\nu_1} F\geqslant\frac{1}{2}$ by a similar argument in the proof of  Proposition
     \ref{billingsley}.
     Let $x\in \kappa(F)\subset \mathcal{G}_{\ell}$. By   the quasi Bernoulli property, we have
     $$C^{-(n+1)}\prod_{i=1}^n\eta_1(\Delta(x_i))\leqslant \eta_1(\Delta(x_1\cdots x_n))\leqslant C^{n+1}\prod_{i=1}^n\eta_1(\Delta(x_i)).$$
 In combination with \eqref{eta-equ-nu} and \eqref{gaussseed}, using once more the fact that $\eta_1$ has density function $\frac{1}{\ln2}\frac{1}{ (1+x)}$, one can show
     $$\lim_{n\to\infty}\frac{\ln\eta_1(\Delta(x_1\cdots x_n))}{n^{3/2}}=-\frac{2}{3}\ln a.$$
          As we have mentioned a fact that the Gauss measure $\eta_1$ is boundedly equivalent to the Lebesgue measure $\lambda$, it follows that (\ref{carjare}) holds.   So by Proposition \ref{carjar} and (\ref{symbolic-r-e}), we have
     $$\dim_H\mathcal{G}_{\ell}\geqslant\dim_H\kappa(F)=\dim_\lambda\kappa(F)=\dim_{\eta_1}\kappa(F)
     =\dim_{\nu_1}F\geqslant\frac{1}{2}.$$
     %which implies $$\dim_HG_\lambda\geqslant\alpha_\nu.$$

     Second, in the case of $|\int\varphi_1 \, \mathrm{d}\theta|<+\infty$,
     consider the set $Y^*$ in the proof of Proposition \ref{beta-lowerbound}.
     Let $x\in\kappa(Y^*)\subset\mathcal{G}_{\ell}$, we can show just like  in the proof of Proposition \ref{beta-lowerbound}
     $$\lim_{n\to\infty}\frac{\ln\lambda(\Delta(x_1\cdots x_n))}{n}=\int_0^12\ln x \, \mathrm{d}\ell(x),$$
     which implies (\ref{carjare}), so by Proposition \ref{carjar}  and (\ref{symbolic-r-e}), we have
    \begin{eqnarray*}
    \dim_H\mathcal{G}_{\ell}&\geqslant&\dim_H\kappa(Y^*)=\dim_\lambda\kappa(Y^*)=\dim_{\eta_1}\kappa(Y^*)\\
    &=&\dim_{\nu_1}Y^*\geqslant\frac{h_\theta}{-\int\varphi_1 \, \mathrm{d}\theta}=\frac{h_\ell}{-2\int_0^1\ln x \, \mathrm{d}\ell(x)}. \end{eqnarray*}
     Combining the two lower bounds,   we have completed  the proof.
     \end{proof}
\subsection{Expanding interval dynamics}
Let $([0,1],f)$ be an expanding interval dynamical system. More precisely, we assume that  $f:[0,1]\to[0,1]$ is a map for which there exists a countable collection of pairwise disjoint open intervals $\{I_a\}_{a\in S}$ such that
\begin{enumerate}
\item [(1)] $[0,1]=\bigcup_{a\in S}\overline{I}_a$;
\item [(2)]$f|_{I_a}$ is a $C^2$ diffeomorphism of $I_a$ onto $(0,1)$ for each $a\in S$;
\item [(3)] Uniform expansion: There exist constants $N\in\mathbb{N}$ and $\xi>1$ such that
$$|(f^N)'(x)|\geqslant\xi\,\,\text{for any }\,x\in\bigcup_{a\in S}{I}_a;$$
\item[(4)] R\'enyi's condition: there exists a constant $K$ such that
$$\sup_{a\in S}\sup_{x,y,z\in I_a}\frac{|f''(x)|}{|f'(y)f'(z)|}\leqslant K.$$

\end{enumerate}
Such dynamic is called Expanding Markov R\'enyi dynamical system (\cite{Pollicott-Weiss}).  Suppose $\ln |f'|$ has summable variations and finite $1$-order variation and ~Gurevich~pressure. For $s\in\mathbb{R}$, consider potential function $\psi_s(x)=-s\ln|f'(x)|$. Set $s_0=\inf\{s\geqslant0:P_{\psi_s}<+\infty\}$. Then for any $s> s_0$ the potential function $\psi_s$ admits a unique Gibbs measure $\eta_s$ and the corresponding convergence exponent will be denoted by $\alpha_{\eta_s}$. Now let us state a  result for the above  expanding interval dynamical systems and omit the proof, because one can transfer the dimension result from the symbolic space to the interval $[0,1]$ as we do with the continued fraction system.
\begin{theorem}
Let $\ell$ be an $f$-invariant Borel probability measure on $[0,1]$. For any $s>s_0$, if $\int_0^1\ln|f'(x)| \, \mathrm{d}\ell(x)<+\infty$ we have
$$\dim_{\eta_s} \mathcal{G}_\ell=\max\left\{\alpha_{\eta_s},\frac{h_\ell}{P_{\psi_s}+s\int_0^1\ln|f'(x)| \, \mathrm{d}\ell(x)}\right\};$$
otherwise we have
$$\dim_{\eta_s} \mathcal{G}_\ell=\alpha_{\eta_s}.$$
\end{theorem}

R\'enyi's condition implies that the Gibbs measure associated to $\psi_1(x)$ is boundedly equivalent to the  Lebesgue measure (\cite{Adler}, see also\cite{Schweiger}, p.~105). As a counterpart of
Theorem \ref{Leb-Haus}  for such general expanding maps, we have the following result.
 \begin{theorem}
             For any $\ell\in \mathcal{M}([0,1), f)$, we have
                  $$\dim_H\mathcal{G}_{\ell}=\max\left\{\alpha_{\eta_1},
                  \frac{h_{\ell}}{\int_0^1\ln|f'(x)|\, \mathrm{d}\ell(x)}\right\}.$$
     \end{theorem}

{\bf{Acknowledgment}}: The authors are supported by Program Caiyuanpei. The second author would like to thank the hospitality of  LAMFA Picarde University where this work is partly done.   The second author is supported by  NSFC(11501112, 11471075) and the third author is supported by  NSFC (11171128, 11271148).


\begin{thebibliography}{10}
\bibitem{Adler}
R. Adler, $F$-expansions revisited. Recent advances in topological dynamics (Proc. Conf., Yale Univ., New Haven, Conn., 1972; in honor of Gustav Arnold Nedlund), pp. 1-5. Lecture Notes in Math., Vol. $\mathbf{318}$, Springer, Berlin, 1973.

\bibitem{B} P. Billingsley, Convergence of probability measures (2th edition). John Wiley \& Sons Inc, 1999.
 \bibitem{bowen}
R. Bowen, Topological entropy for noncompact set. Trans. Amer. Math. Soc., $\mathbf{184}$ (1973), 125-136.
 \bibitem{cajar}
H. Cajar, Billingsley Dimension in probability spaces. Springer-Verlag, 1981.
\bibitem{DGS}
M. Denker, C. Grillenberger, K. Sigmund, Ergodic theory on compact spaces. Lecture Notes in Mathematics, Vol. $\mathbf{527}$. Springer-Verlag, Berlin-New York, 1976.
\bibitem{FAN-FENG}
A. H. Fan, D. J. Feng, On the distribution of long-term time averages on symbolic space. J. Statist. Phys. $\mathbf{99}$ (2000), no. 3-4, 813-856.


\bibitem{FFW}
A. H. Fan, D. J. Feng and J. Wu, Recurence, dimension and entropy.   J. London Math. Soc. (2) $\mathbf{64}$ (2001), 229-244.

\bibitem{FLM}
A. H. Fan, L. M. Liao and J. H. Ma, On the frequency of partial quotients of regular continued fractions. Math. Proc. Camb. Phil. Soc. $\mathbf{148}$ (2010), 179-192.
\bibitem{FLMW}
 A. H. Fan, L. M. Liao, J. H. Ma and B. W. Wang, Dimension of Besicovitch-Eggleston sets in countable symbolic space. Nonlinearity $\mathbf{23}$ (2010) 1185-1197.
\bibitem{FLP}
A. H. Fan, L. M. Liao and J. Peyri\`ere, Generic points in systems of specification and Banach valued Birkhoff ergodic average.
 Discrete Contin. Dyn. Syst. $\mathbf{21}$ (2008), no. 4, 1103-1128.
\bibitem{FJLR}
A. H. Fan, T. Jordan, L.M. Liao and M. Rams, Multifractal analysis for expanding interval maps with infinitely many branches. Trans. Amer. Math. Soc. $\mathbf{367}$ (2015), no. 3, 1847¨C1870.

 \bibitem{GW}
E. Glasner and  B.  Weiss, On the interplay between measurable and topological dynamics, Handbook of dynamical systems. Vol.1B (Elsevier B.V., Amsterdam, 2006), 597-648.
 %\bibitem{LMW} Lingmin Liao, Jihua Ma, Baowei Wang, Dimension of some non-normal continued fraction Sets
 \bibitem{GT}
 B. M. Gurevich, A. A. Tempelman, Hausdorff dimension of sets of generic points for Gibbs measures. Dedicated to David Ruelle and Yasha Sinai on the occasion of their 65th birthdays. J. Statist. Phys. $\mathbf{108}$ (2002), no. 5-6, 1281-1301.

\bibitem{LMW}
L. M. Liao, J.H. Ma and B.W. Wang, Dimension of some non-normal continued fraction sets. Math. Proc. Cambridge Philos. Soc. $\mathbf{145}$ (2008), no. 1, 215-225.

\bibitem{Mawen}
J. H. Ma, Z. Y. Wen, Hausdorff and packing measure of sets of generic points: a zero-infinity law. J. London Math. Soc. (2) $\mathbf{69}$ (2004), 383-406
\bibitem{Mat}
P. Mattila, Geometry of sets and measures in Euclidean spaces. Fractals and rectifiability. Cambridge Studies in Advanced Mathematics, $\mathbf{44}$. Cambridge University Press, Cambridge, 1995.
\bibitem{Olivier}
E. Olivier, Dimension de Billingsley d'ensembles satur\'es.  C. R. Acad. Sci. Paris S\'er. I Math. $\mathbf{328}$ (1999), no. 1, 13-16.

\bibitem{Par}
K. R. Parthasarathy, Probability measure on metric space. Academic Press, New York and London, 1967.

\bibitem{PS1}
C.-E.  Pfister, W. G.  Sullivan, Billingsley dimension on shift spaces. Nonlinearity $\mathbf{16}$ (2003), no. 2, 661-682.
\bibitem{PS2}
C.-E.  Pfister, W. G.  Sullivan,
On the topological entropy of saturated sets. Ergodic Theory Dynam. Systems $\mathbf{27}$ (2007), no. 3, 929-956.


\bibitem{Pollicott-Weiss}
M. Pollicott, H.  Weiss, Multifractal analysis of Lyapunov exponent for continued fraction and Manneville-Pomeau transformations and applications to Diophantine approximation. Comm. Math. Phys. $\mathbf{207}$ (1999), no. 1, 145-171.



\bibitem{PS}
 G. P\'olya, G. Szeg\"o, Problems and theorems in analysis. Vol. I: Series, integral calculus, theory of functions. Springer, 1972.
%\bibitem{PU}
%F. Przytycki, M. Urbanski, Conformal fractals: ergodic theory methods.  Cambridge University Press, Cambridge 2010.
\bibitem{Sarig0}
O. Sarig, Thermodynamic formalism for countable Markov shifts.  Ergodic Theory Dynam. Systems $\mathbf{19}$ (1999), no. 6, 1565-1593.
\bibitem{Sarig1}
O. Sarig, Existence of Gibbs measures for countable Markov shifts. Proc. Amer. Math. Soc. $\mathbf{131}$ (2003), no. 6, 1751-1558.
\bibitem{Sarig2}
O. Sarig, Lecture notes on thermodynamic formalism for topological markov shifts. 2009. See also http://www.wisdom.weizmann.ac.il/~sarigo/TDFnotes.pdf.
\bibitem{Schweiger}
F. Schweiger,  Ergodic theory of fibred systems and metric number theory. Oxford Science Publications. The Clarendon Press, Oxford University Press, New York, 1995.


\bibitem{Sen}
E. Seneta, Non-negative matrices and Markov chains. Springer, 2006.
\bibitem{walters1978}
P. Walters, Invariant measures and equilibrium states for some mappings which expand distances. Trans. Amer. Math. Soc. $\mathbf{236}$ (1978), 121-153.

 \bibitem{walters}
P. Walters, An introduction to ergodic theory. Springer-Verlag,  1982.
\bibitem{weg}
H. Wegmann, \"Uber den dimensionsbegriff in wahrscheinlichkeitsrumen, II Z. Wahrscheinlichkeitstheor. Verwandte Geb. $\mathbf{9}$ (1968), 222-231.

\end{thebibliography}
\end{document}